\documentclass{amsart}
\usepackage{amsmath}
\usepackage{amssymb}
\usepackage{amsthm}
\usepackage{epsfig}
\usepackage{enumitem}

\newcommand{\amsprimary}[1]{{\footnotesize\noindent AMS 2010 \textit{Mathematics subject
classification:} Primary #1\vspace{1pc}}}
\newcommand{\keywordsnames}[1]{{\footnotesize\noindent\textit{Key words:} #1\vspace{1pc}}}

\title[Ricci flow on surfaces with boundary]{The Ricci flow on surfaces
with boundary}
\newtheorem{theorem}{Theorem}[section]
\newtheorem{definition}{Definition}[section]
\newtheorem{proposition}{Proposition}[section]
\newtheorem{lemma}{Lemma}[section]
\newtheorem{corollary}{Corollary}[section]

\author{Jean C. Cortissoz
\and Alexander Murcia}
\address{Departamento de Matem\'aticas, Universidad de los Andes, Bogot\'a DC  (Colombia);
\quad  Universidad Santo Tom\'as, Bogot\'a DC  (Colombia)}

\begin{document}
\maketitle
\begin{abstract}
We study a boundary value problem
for the Ricci flow on a surface with boundary, where the geodesic curvature
of the boundary is prescribed. 
\end{abstract}
{\keywordsnames {Ricci flow; surfaces with boundary.}}

{\amsprimary {53C44; 35K55.}}

\section{Introduction}

The Ricci flow on surfaces, compact and noncompact, has been an intense subject of study since the appearance of  
Hamilton's seminal work \cite{Hamilton0}, where the asymptotic 
behavior of the flow is studied on closed surfaces, and it is used as a tool towards giving a proof 
of the Uniformization Theorem via parabolic methods. In addition to its obvious geometric appeal, it is of note
that the study of the Ricci flow on surfaces is related to the study of the logarithmic diffusion equation, and hence, the
interest on this problem goes beyond its geometric applications (see \cite{Hsu}).

However,
not much is known about the behavior of the Ricci flow on manifolds with boundary. One of the main difficulties
in studying this problem
arises from the fact that even trying to impose meaningful boundary conditions
for the Ricci flow, for
which existence and uniqueness results can be proved so interesting geometric
applications can be hoped for, seems to be a challenging task.
For the reader to get an idea of the difficulty of the problem, we recommend the interesting works of Y. Shen 
\cite{Shen}, S. Brendle \cite{Brendle},
A. Pulemotov \cite{Pulemotov} and P. Gianniotis \cite{Gianniotis}. In the case of the
boundary conditions imposed by Shen \cite{Shen}, satisfactory convergence
results have given for manifolds of positive Ricci curvature
and totally geodesic boundary, and also when the boundary is convex
and the
metric is rotationally symmetric. In the case of surfaces, the Ricci flow is parabolic, and imposing
natural geometric boundary conditions is not difficult: one can for instance control the geodesic curvature
of the boundary. In this case, Brendle \cite{Brendle} has shown that when the boundary is totally 
geodesic, then the behavior is completely analogous to the behavior of the Ricci flow in closed 
surfaces (\cite{Hamilton0,Chow0}). In this case, also for non totally geodesic boundary, the first author has proved,
under the hypothesis of rotational symmetry of the metrics involved, 
results on the asymptotic behavior of the Ricci flow in the case of positive curvature
and convex boundary, and for certain families of metrics with non convex boundary (\cite{Cortissoz1}). 

The purpose of this paper 
is to contribute towards the understanding of the behavior of the Ricci flow 
on surfaces with boundary. To 
be more precise, let $M$ be a compact surface with boundary ($\partial M\neq \emptyset$),
endowed with a smooth metric $g_0$; 
we will study the equation
\begin{equation}
\label{flow}
\left\{
\begin{array}{l}
\frac{\partial g}{\partial t}=-R_g g \quad\mbox{in}\quad M\times\left(0,T\right)\\
k_g\left(\cdot,t\right)=\psi\left(\cdot,t\right) \quad \mbox{on} \quad \partial M \times\left(0,T\right)\\
g\left(\cdot,0\right)=g_{0}\left(\cdot\right)\quad \mbox{in}\quad M,
\end{array}
\right.
\end{equation}
where $R_g$ represents the scalar curvature of $M$ and $k_g$ the geodesic curvature of $\partial M$, both
with respect to the time evolving metric $g$, and $\psi$ is a smooth real valued function defined on
$\partial M\times \left[0,\infty\right)$, and which satifies the compatibility condition 
$\psi\left(\cdot,0\right)=k_{g_0}$.

The short-time existence theory of equation (\ref{flow}) is well understood. Indeed, since the 
deformation given by (\ref{flow}) is conformal, if we write $g\left(p,t\right)=e^{u\left(p,t\right)}g_0$,
problem (\ref{flow}) is equivalent to a nonlinear parabolic equation with Robin boundary conditions,
and initial condition $u\left(p,0\right)=1$. Hence, via the Inverse Function Theorem and standard
methods from the theory of parabolic equations \cite{Lady},
it can be shown that (\ref{flow}) has a unique solution for a short time, and that this solution
is in the parabolic
H{\"o}lder space  $H^{2+\alpha,1+\frac{\alpha}{2}}$, $0<\alpha<1$, on $\overline{M}\times \left[0,T\right)$, and smooth away from the corner.
 
Before 
we state the results we intend to prove in this paper, we must introduce a normalization of (\ref{flow}).
As it is well known, the solution to (\ref{flow}) can be normalized to keep the area of the surface constant. This is done as follows:
Let us assume without loss of generality that the area of $M$ with respect to $g_0$ is $2\pi$, and 
choose $\phi\left(t\right)$ such that $\phi\left(t\right) A_g\left(t\right)=2\pi$, 
where $A_g\left(t\right)$ is the area of the surface
at time $t$ with respect to the metric $g$. Then define
\begin{equation}
\label{normalization}
\tilde{t}\left(t\right)=\int_0^t \phi\left(\tau\right)\,d\tau \quad \mbox{and}
\quad \tilde{g}=\phi g.
\end{equation}
If the family of metrics $g\left(t\right)$ satisfies (\ref{flow}), then 
the family of metrics $\tilde{g}\left(\tilde{t}\right)$ satisfies the evolution equation
\begin{equation}
\label{normalizedflow}
\left\{
\begin{array}{l}
\frac{\partial \tilde{g}}{\partial \tilde{t}}=\left(\tilde{r}_{\tilde{g}}-\tilde{R}_{\tilde{g}}\right)\tilde{g} 
\quad\mbox{in}\quad M\times\left(0,\tilde{T}\right)\\
k_{\tilde{g}}\left(\cdot,t\right)=\tilde{\psi}\left(\cdot,\tilde{t}\right) 
\quad \mbox{on} \quad \partial M\times\left(0,\tilde{T}\right)\\
\tilde{g}\left(\cdot,0\right)=g_{0}\left(\cdot\right)\quad\mbox{on}\quad M,
\end{array}
\right.
\end{equation}
where $\tilde{\psi}$ is the normalization of the function $\psi$, $\tilde{R}_{\tilde{g}}$ is the scalar curvature of the metric $\tilde{g}$, and
\[
\tilde{r}_{\tilde{g}}=\frac{\int_{M}\tilde{R}_{\tilde{g}}\,dA_{\tilde{g}}}{\int_{M}\,dA_{\tilde{g}}}
=\frac{1}{2\pi}\int_M\tilde{R}_{\tilde{g}}\,d\tilde{A}_g.
\]
Here $dA_{\tilde{g}}$ denotes
the area element of $M$ with respect to the metric $\tilde{g}$.
We refer to (\ref{normalizedflow}) as the {\it normalized Ricci flow}.

We can now state our first result.
\begin{theorem}
\label{maintheorem}
Let $\left(M^2,g_0\right)$ be a compact surface with boundary with positive scalar curvature
($R_{g_0}>0$), and such that the geodesic curvature of its boundary 
is nonnegative ($k_{g_0}\geq 0$), and assume that $\psi$, as defined above,  is nonnegative and also 
satisfies that $\frac{\partial}{\partial t}\psi\leq 0$.
Let $g\left(t\right)$ be the solution to (\ref{flow}) with initial condition
$g_0$. Then the corresponding solution to the normalized flow, $\tilde{g}\left(\tilde{t}\right)$, exists for all time, and 
for any sequence $\tilde{t}_n\rightarrow \infty$, 
there is a subsequence $\tilde{t}_{n_k}\rightarrow \infty$ such that
the metrics $\tilde{g}\left(\tilde{t}_{n_k}\right)$
converge smoothly to a metric of constant curvature and totally geodesic boundary.
\end{theorem}

Theorem \ref{maintheorem} partially extends results on the asymptotic behavior of solutions
to the Ricci flow known for $\mathbb{S}^2$ (Hamilton \cite{Hamilton} and Chow \cite{Chow0}),
and for the case of surfaces with totally geodesic boundary (Brendle \cite{Brendle}) and
with rotational symmetry (\cite{Cortissoz1}).
 
Before going any further, let us give an outline of the proof of Theorem \ref{maintheorem}. First of all,
given an initial metric $g_0$ of positive scalar curvature and
convex boundary, it can be shown that the curvature of $g\left(t\right)$
blows up in finite time, say $T<\infty$; the idea then is to take a blow up limit 
of the solution $\left(M,g\left(t\right)\right)$ as $t\rightarrow T$, and to show that the only
possibility for this blow up limit is to be a round hemisphere: this would, essentially, give
a proof of Theorem \ref{maintheorem}. It remains then to remove one
technical difficulty: we must be able to produce this blow up limit, and hence
we will have to show that we can estimate the injectivity radius of the surface, and, because 
it has boundary, we are required to show that there are no geodesics hitting the boundary
orthogonally that are too short with respect
to the inverse of the square root
of the maximum of the curvature. This is the basic new ingredient in the proof of 
Theorem \ref{maintheorem}, and to prove it we have introduced  
  an extension procedure
for surfaces with boundary that allows some control over the maximum curvature and
size of the extension.  

Our second result is concerned with the behavior of the Ricci flow when the geodesic curvature of the boundary is nonpositive.
Again, using blow up analysis techniques, we prove the following theorem, which generalizes similar results from \cite{Cortissoz1} (notice that we make
no requirements on the sign of $R$).
\begin{theorem}
\label{maintheorem2}
Let $g_0$ be a rotationally symmetric metric on the two-ball. Assume that $k_{g_0}\leq 0$, and
that the boundary data is given by $\psi=k_{g_0}$. Then the normalized flow corresponding
to the solution to (\ref{flow}) with initial data $g_0$ and boundary data $\psi$ exists for 
all time.   
\end{theorem}

The layout of this paper is as follows. In Section \ref{evolution} we prove the basic evolution 
equation for the scalar curvature when the metric evolves under (\ref{flow}), and show that
under certain conditions the curvature $R$ blows up in finite time; in Section
\ref{Perelmanmonotonicity} we prove a monotonicity formula for Perelman's functionals 
on surfaces with boundary; in Section \ref{Non-collapsing} it is shown that it is possible to take
blow up limits for solutions to (\ref{flow}), by proving that
we can control the injectivity radius of the surface in terms of the scalar curvature
and the geodesic curvature of the boundary, and proving a compactness theorem for sequences of Ricci
flows; 
in Section \ref{smoothconvergence} we use the results from the previous sections to give a proof of
Theorem \ref{maintheorem}. Finally, in Section \ref{neggeod} we give
a proof of Theorem \ref{maintheorem2}. This paper is complemented by an appendix where among
other things we discuss a procedure to obtain bounds on the derivatives of solutions to (\ref{flow}) -and hence to (\ref{normalizedflow})- in terms
of bounds on the curvature and the boundary data (and its derivatives).

\medskip
Parts of this paper are part of the PhD Thesis of the second author. He wants to thank his advisor 
(the first named author of this paper) and  his home institution (while completing his PhD), Universidad
de los Andes, for their support and encouragement during his studies.

\section{Evolution equations}
\label{evolution}

In the following proposition, which is stated in \cite{Cortissoz1} without proof, we compute the evolution of the curvature of 
a metric $g$ when it is evolving under (\ref{flow}).

\begin{proposition}
\label{evolutionscalarcurv}
Let $\left(M,g\left(t\right)\right)$ be a solution to (\ref{flow}).
The scalar curvature satisfies the evolution equation
\[
\left\{
\begin{array}{l}
\frac{\partial R_g}{\partial t}=\Delta_g R_g+R_g^2 \quad \mbox{in}\quad M\times\left(0,T\right)\\
\frac{\partial R}{\partial \eta_g}= k_g R_g-2k_g' = \psi R_g -2\psi' \quad \mbox{on}\quad \partial M\times\left(0,T\right)
\end{array}
\right.
\]
where $\eta_g$ is the outward pointing unit normal with respect to the metric $g$, and the
prime ($'$) represents differentiation with respect to time.
\end{proposition}
\begin{proof}
Since the evolution equation satisfied by $R$ in the interior of $M$ is known (\cite{Hamilton}), we will just compute
its normal derivative, with respect to the outward normal, at the boundary. To do so,
we choose local coordinates $\left(x^1,x^2\right)$ at $p\in \partial M$ such that $x^2=0$ is a defining
function for $\partial M$, so that the corresponding coordinate frame
$\left\{\partial_1,\partial_2\right\}$ is orthonormal at $p\in \partial M$ and time  $t=t_0$
(i.e., the point and instant when we want to compute the normal derivative), and so that  $\partial_2$ 
coincides with the outward unit normal to the boundary
in the whole coordinate patch (this also at time $t=t_0$). 
Since the deformation is conformal, $\partial_2$ remains normal to the boundary.
Therefore the geodesic curvature is given  (as long as the flow is defined for $t\geq t_0$) by
the formula
\[
k_g g_{11} =-\frac{\Gamma_{11}^2}{\left(g^{22}\right)^{\frac{1}{2}}} =-\left(g_{22}\right)^{\frac{1}{2}}\Gamma_{11}^2.
\]
Computing the time derivative, the previous identity yields
\begin{eqnarray*}
\left(k_g g_{11}\right)' &=&  -\frac{1}{2\left(g_{22}\right)^{\frac{1}{2}}}\left(g_{22}\right)'\Gamma_{11}^2
-\left(g_{22}\right)^{\frac{1}{2}}\left(\Gamma_{11}^2\right)'\\
&=&\frac{1}{2}R_g\left(g_{22}\right)^{\frac{1}{2}}\Gamma_{11}^2-\left(g_{22}\right)^{\frac{1}{2}}\left(\Gamma_{11}^2\right)'.
\end{eqnarray*}
Let us calculate $\left(\Gamma_{11}^2\right)'$ (as is customary $\nabla_j$ denotes covariant differentiation
with respect to $\partial_j$, and recall that $g_{12}=0$ and $g_{ii}=1$)
\begin{eqnarray*}
\left(\Gamma_{11}^2\right)'&=& \frac{1}{2}g^{2j}\left(\nabla_1 g_{1j}'+\nabla_1 g_{1j}'-\nabla_j g_{11}'\right)\\
&=&\frac{1}{2}g^{22}\left(-2\nabla_1 \left(R_gg_{12}\right)+\nabla_2\left(R_gg_{11}\right)\right)\\
&=&\frac{1}{2}g^{22}\left(\partial_2 R_g\right)g_{11}=\frac{1}{2}g^{22}\left(\partial_2 R_g\right).
\end{eqnarray*}
Therefore
\begin{equation*}
k_g'g_{11}-k_g R g_{11}=-\frac{1}{2}k_g R - \frac{1}{2\left(g_{22}\right)^{\frac{1}{2}}}\partial_2 R_g
=-\frac{1}{2}k_g R_g -\frac{1}{2}\frac{\partial R_g}{\partial \eta_g},
\end{equation*}
and the result follows.
\end{proof}

As a consequence from Hopf Maximum Principle, since $k_g'=\psi'\leq 0$ in the case
we are considering, we obtain the following result.
\begin{proposition}
\label{blowupprop}
Let $\left(M,g\left(t\right)\right)$, $M$ compact, be a solution to (\ref{flow}).
Assume that $\psi$, the boundary data, satisfies $\psi'\leq 0$.
Then, if $R_{g}\geq 0$ at time $t=0$, it remains so as long as the solution exists.
Furthermore, if the initial data has positive scalar curvature and
the boundary data $\psi$ is nonnegative, then  $R_g$ remains strictly positive and blows up in finite time.
\end{proposition}
\begin{proof}
We leave the proof that $R$ remains strictly positive to the reader, and show that the solution
of (\ref{flow}) must blow-up in finite time. By Hopf Maximum Principle, since 
by the hypotheses at the boundary
we have 
\[
\dfrac{\partial R}{\partial \eta}\geq 0,
\]
the minimum of $R_{\mbox{min}}\left(t\right)$ of $R$ at time $t$ occurs in the interior of $M$. Hence, $R_{\mbox{min}}$
satisfies a differential inequality
\[
\dfrac{d}{dt}R_{\mbox{min}}\geq R_{\mbox{min}}^2.
\] 
Therefore, comparing with the solution of the ODE
\[
\dfrac{du}{dt}=u^2, \quad u\left(0\right)=R_{\mbox{min}}\left(0\right),
\]
we have that $R_{\mbox{min}}\geq u$, and since $u>0$, $u$ must blow up in finite time, and so must $R_{\mbox{min}}$. 
\end{proof}

We must point out that if $R\geq 0$ at time $t=0$, and it is strictly positive at a point, under the assumption
$\psi\geq 0$, $\psi'\leq 0$, it becomes strictly positive instantaneously, so the hypotheses in the
previous proposition may be relaxed a bit.
\subsection{}
In view of Proposition \ref{blowupprop}, this seems a good place to discuss the following fact.
Let $\left(0,T\right)$ be the maximal interval of existence of a solution to (\ref{flow}), with $0<T<\infty$, then
\[
\limsup_{t\rightarrow T}\left(\sup_{p\in M}R_g\left(p,t\right)\right)=\infty.
\]
First of all if $g_0$ is the initial metric, then as the Ricci flow preserves conformal structure, we have that 
the evolving metric can be represented as $g=e^{u}g_0$. Hence, if $R_{g_0}$ is the scalar curvature of the initial metric,
at a fixed (but arbitrary) time, we have that $u$ satisfies the elliptic boundary value problem
\begin{equation}
\label{conformaldefform}
\left\{
\begin{array}{l}
\Delta_{g_0}u+R_{g_0}=R_ge^{u} \quad \mbox{in}\quad M\\
\frac{\partial}{\partial \eta_{g_0}}u+2k_{g_0}=2k_{g_0} e^{\frac{u}{2}}
\quad \mbox{on}\quad \partial M.
\end{array}
\right.
\end{equation}
To reach a contradiction assume that  $R_g$ remains uniformly bounded on $\left(0,T\right)$. A
consequence of this assumption is
that $e^{u}$  remains bounded away from $0$ and uniformly bounded above
on $\left(0,T\right)$.
 Now, since
from bounds on the curvature and also on the geodesic curvature of the boundary
(i.e., on $\psi$)  and its derivatives, we can obtain 
bounds on the derivatives of $u$ (see Theorems \ref{firstderivative} and \ref{derivativesconformalfactor} in the Appendix),  $u$ and its derivatives 
(including those with respect to $t$) are
uniformly bounded on $\left(0,T\right)$, and consequently they converge as $t\rightarrow T$ to a smooth function, say $\hat{u}$.
If we start the Ricci flow at $t=T$ with initial data $e^{\hat{u}}g_0$ and the same boundary data, then we would be able to continue the original
solution past $T$, which contradicts the hypothesis. Therefore, if the Ricci flow (\ref{flow}) cannot be extended past $T<\infty$, the curvature
blows up. We invite the reader to consult the recent work of Gianniotis \cite{Gianniotis}, where this property of the Ricci flow on
manifolds with boundary is discussed in a more general context.

\subsection{} 
There are other interesting cases when solutions to (\ref{flow}) blow-up. We have for instance the following proposition.
\begin{proposition}
\label{otherblowups}
Assume that $\int_M R_{g_0}\, dA_{g_0}+\int_{\partial M}2k_{g_0}\,ds_{g_0}>0$, and assume that $\psi\leq 0$. Then the solution to (\ref{flow})
with initial condition $g_0$ and boundary data $\psi$, blows up in finite time.
\end{proposition}
\begin{proof}
Let $g\left(t\right)$ be the solution to (\ref{flow}) with initial data $g_0$ and boundary data $\psi$.
If $A\left(t\right)$ represents the area of $M$ with respect to $g\left(t\right)$, we can calculate
\[
\frac{dA}{dt}=-\int_{M}R_g\,dA_{g}=-4\pi\chi\left(M\right)+2\int_{\partial M}k_{g}\,ds_{g}\leq -4\pi\chi\left(M\right).
\]
Therefore, the area cannot be positive for all time, hence 
 a singularity must occur in finite time.
\end{proof}
\section{Monotonicity of Perelman's Functionals on surfaces with boundary.}
\label{Perelmanmonotonicity}

The purpose of this section is to show a monotonicity formula for Perelman's celebrated $\mathcal{F}$
and $\mathcal{W}$ functionals (see \cite{Perelman})
in the case of surfaces with boundary. The results in this section are stated, although with no carefully crafted proofs,
more or less in the same way in
\cite{Cortissoz1}.
As usual, all curvature quantities, scalar products and operators depend on the time-varying metric $g$
(some of them will not bear a subindex to show that dependence). We  
will use the Einstein summation convention freely, and the raising and lowering of indices
is done, by means of the metric $g$, in the usual way. 

In order to proceed, recall the definition of  Perelman's  $\mathcal{F}$-functional:   
\begin{eqnarray*}
\label{perelman1}
\mathcal{F}\left(g_{ij},f\right)&=& \int_{M} \left(R_g+\left|\nabla f\right|^2\right)
\exp\left(-f\right)\,dV_g, 
\end{eqnarray*}
where $dV_g$ represents the volume (in the case of a surface, area) element of the manifold
$M$ with respect to the metric $g$. 
Let us compute the first variation of this functional on a manifold with boundary.

\begin{proposition}
\label{firstvariation1}
Let $\delta g_{ij}=v_{ij}, \delta f = h, g^{ij}v_{ij}=v$. Then we have,
\begin{eqnarray*}
\label{firstvariation}
\delta \mathcal{F}
&=&\int_M\exp\left(-f\right)\left[-v^{ij}\left(R_{ij}+\nabla_i\nabla_j f\right)
+\right(\frac{v}{2}-h\left)\left(2\Delta_g f-\left|\nabla f\right|^2+R_g\right)\right]\,dV_g \notag\\
&&-\int_{\partial M}\left[\frac{\partial v}{\partial \eta_g} +\left(v-2h\right)\frac{\partial f}{\partial \eta_g}\right]
\exp\left(-f\right)\,d\sigma_g+\\
&&\int_{\partial M}\exp\left(-f\right)\nabla_i v_{ij}\eta^j\,d\sigma_g
-\int_{\partial M} \nabla_j\exp\left(-f\right)v_{ij}\eta^i\,d\sigma_g .\notag
\end{eqnarray*}
Here, $R_{ij}$ represents the Ricci tensor of the metric $g$, 
$\displaystyle\frac{\partial}{\partial \eta_g}$ ($=\eta^i\partial_i$ in local coordinates) is the outward unit normal to $\partial M$ with respect to $g$, 
$\nabla$ represents
covariant differentiation
with respect to the metric $g$, and $d\sigma_g$ represents the volume element of $\partial M$.
\end{proposition}
\begin{proof}
As in \cite{KleinerLott}, we have that
\begin{eqnarray*}
\delta\mathcal{F}\left(v_{ij},h\right)&=&
\int_{M}e^{-f}\left[-\Delta_g v+\nabla_i\nabla_jv^{ij}-R_{ij}v^{ij}\right.\\
&&\left.-v^{ij}\nabla_if\nabla_jf+2g\left(\nabla f,\nabla h\right)
+\left(R_g+\left|\nabla f\right|^2\right)\left(\frac{v}{2}-h\right)\right]\,dV_g.
\end{eqnarray*}
We must compute the integrals on the righthand side of the previous identity, using as our
main tool integration by parts.
We start by calculating
\begin{eqnarray*}
\int_M e^{-f}\left(-\Delta_g v\right)\,dV_g &=& 
-\int_M \Delta_g e^{-f}v\,dV_g+\int_{\partial M}v\frac{\partial e^{-f}}{\partial \eta_g}\,d\sigma_g
-\int_{\partial M}e^{-f}\frac{\partial v}{\partial \eta_g}\,d\sigma_g\\
&=&-\int_M \Delta_g e^{-f}v\,dV_g
-\int_{\partial M}\left(\frac{\partial v}{\partial \eta_g}+v\frac{\partial f}{\partial \eta_g}\right)\exp\left(-f\right)\,
d\sigma_g.
\end{eqnarray*}
Now we compute
\begin{eqnarray*}
\int_M e^{-f}\nabla_i\nabla_j v^{ij}\,dV_g &=& -\int_M \nabla_i e^{-f}\nabla_j v^{ij}\,dV_g
+\int_{\partial M} e^{-f}\nabla_j v^{ij} \eta^i\,d\sigma_g \\
&=& \int_M \nabla_i\nabla_j e^{-f} v^{ij}\,dV_g
-\int_{\partial M} \nabla_i e^{-f} v^{ij} \eta_j\,d\sigma_g\\
&&+\int_{\partial M} e^{-f}\nabla_j v^{ij} \eta_i\,d\sigma_g.
\end{eqnarray*}
Finally,
\begin{eqnarray*}
2\int_M e^{-f}g\left(\nabla f,\nabla h\right) \,dV_g&=& -2\int_M g\left(\nabla e^{-f},\nabla h\right)\,dV_g\\
&=&2\int_M \left(\Delta_g e^{-f}\right) h\,dV_g-\int_{\partial M}h\frac{\partial e^{-f}}{\partial \eta_g}\,d\sigma_g.
\end{eqnarray*}
Putting all these calculations together proves the result.
\end{proof}

Consider the evolution equations on a surface with boundary given by
\begin{equation}
\label{evolution1}
\left\{
\begin{array}{l}
\frac{\partial}{\partial t}g_{ij} = -R_gg_{ij}=-2R_{ij}\quad\mbox{in}\quad M\times\left(0,T\right)\\
k_g\left(\cdot,t\right)=\psi\left(\cdot\right) \quad\mbox{on}\quad \partial M\times\left(0,T\right)\\
\frac{\partial f}{\partial t}=-\Delta_g f+\left|\nabla f\right|^2 -R_g \quad\mbox{in}\quad M\times\left(0,T\right)\\
\frac{\partial}{\partial \eta_g}f=0\quad\mbox{on}\quad \partial M\times\left(0,T\right).
\end{array}
\right.
\end{equation}
A formula for $\dfrac{d}{dt}\mathcal{F}$ is given by the following result.
\begin{theorem}
\label{monotonicity1}
Under (\ref{evolution1}) the functional $\mathcal{F}$ satisfies
\begin{eqnarray*}
\frac{d}{dt}\mathcal{F}&=& 2\int_M \left|R_{ij}+\nabla_i\nabla_j f\right|^2\exp\left(-f\right)\,dA_g\\
&&+\int_{\partial M}\left(k_g R_g-2k_g'\right)\exp\left(-f\right)\,ds_g 
+2\int_{\partial M}k_g\left|\nabla^{\top}f\right|^2\exp\left(-f\right)\,ds_g,
\end{eqnarray*}
and here $\nabla^{\top} f$ represents the component of  $\nabla f$
tangent to $\partial M$, $dA_g$ the area element of the surface, and
$ds_g$ the length element of the boundary.
\end{theorem}

\begin{proof}
Let us first introduce some notation and conventions. Since
parts of these computations apply to manifolds 
of higher dimensions, in this proof we will 
fix coordinates $x^1,x^2,\dots,x^{n-1},x^n$ at a boundary point and at 
fixed (but arbitrary) time $t$, so that $x^n=0$ is a defining function for $\partial M$. 
We will assume that on $\partial M$, $\displaystyle\frac{\partial}{\partial x^n}=\frac{\partial}{\partial\eta_g}$ represents
the outward unit normal, and hence we will
denote by a subscript or superscript $n$ quantities that 
are evaluated, at a boundary point, with respect to the outward unit normal. By a greek letter we
will represent indices running from $1,2,3,\dots, n-1$, and therefore at a boundary point 
the vector fields $\displaystyle\frac{\partial}{\partial x^{\alpha}}$ are tangent to the boundary. 
Let us transform the evolution equations given by (\ref{evolution1}) using the one-parameter family
of diffeomorphisms $\varphi_t$ generated by $-\nabla f$; notice that
the boundary is sent to itself via this family of diffeomorphisms due to the
fact that $\displaystyle \frac{\partial f}{\partial \eta_g}=0$. Now, by defining $f\left(\cdot,t\right)=f\left(\varphi_{t}\left(\cdot\right),t\right)$ and 
$g=\left(\varphi_t\right)_{*}g$ (forgive the abuse of notation),
instead of (\ref{evolution1}) we must take the variations given by
\[
v_{ij}=\delta g_{ij}=-2\left(R_{ij}+\nabla_i\nabla_j f\right), \quad h=\delta f=-\Delta_g f -R.
\]
For a moment let us denote with a subindex $\left(\varphi_t\right)_{*}g$ the quantities that
depend on the pullback metric. Observe that then we have
\[
\frac{\partial}{\partial \eta_{\left(\varphi_t\right)_{*}g}}R_{\left(\varphi_t\right)_{*}g}\left(\cdot,t\right)
=\left(k_g R_g-2k_g'\right)\left(\varphi_t\left(\cdot\right),t\right),
\] 
so keeping on with the abuse of notation, we will write, 
for the metric $g=\left(\varphi_t\right)_{*}g$,
\[
\frac{\partial R}{\partial \eta_g}=k_g R_g - 2k_g',
\]
where the prime ($'$) now means that we differentiate $k_g$ (or $\psi$) with respect to its second variable ($t$) and then 
it is evaluated at $\left(\varphi_t\left(\cdot\right),t\right)$.
Notice that for the pullback metric we still have $\displaystyle\frac{\partial f}{\partial \eta_g}=0$.

We will now compute each of the boundary integrals in the first variation of Perelman's
functional given by Proposition \ref{firstvariation1}, which will prove the theorem, since the computations 
for the integrals over $M$ are known from
the work of Perelman. We change the notation from the previous proposition as follows: $dV_g=dA_g$ and $d\sigma_g=ds_g$.
We start with
\[
\int_{\partial M}\left[\frac{\partial v}{\partial \eta_g} +\left(v-2h\right)\frac{\partial f}{\partial \eta_g}\right]
\exp\left(-f\right)\,ds_g, \quad
\int_{\partial M}\exp\left(-f\right)\nabla_i v^{ij}\eta^j\,ds_g.
\]
To compute these integrals, let us first calculate 
$\nabla_i v^{ij}\eta_j$. We have
\begin{eqnarray*}
\nabla_i v_{\,\,n}^i&=&-2\nabla_i R_{\,\,n}^i-2\nabla_i\nabla^i\nabla_n f\\
&=&-\nabla_n R_g-2\Delta_g \nabla_n f \quad\mbox{(by the contracted Bianchi identity)}.
\end{eqnarray*}
By the Ricci identity, using the fact that at the boundary $\nabla_n f=0$ 
and also $R_{\alpha n}=0$, we obtain
\[
\Delta_g \nabla_n f = \nabla_n\Delta_g f+R_{n}^{\,\,\,k}\nabla_k f= \nabla_n \Delta_g f,
\]
and therefore,
\[
\nabla_i v_{\,\,n}^i=-\nabla_n R_g - 2\nabla_n \Delta_g f.
\]
Using the evolution equation $f_t=-\Delta_g f - R$, we get, at the boundary,
\[
\nabla_n\Delta_g f = -\nabla_n R_g.
\]
This last identity has two consequences. 
On the one hand, it implies that
\[
\nabla_i v_{\,\,n}^i=-\nabla_n R_g+2\nabla_n R=\nabla_n R_g=k_g R_g-2k_g',
\]
which shows that
\[
\int_{\partial M}\exp\left(-f\right)\nabla_i v^{ij}\eta_j\,ds_g=\int_{\partial M}
\left(k_g R_g-2k_g'\right)\exp\left(-f\right)\,ds_g.
\]
On the other hand, it implies
that $\frac{\partial v}{\partial \eta_g}=0$, so
we obtain
\[
\int_{\partial M}\left[\frac{\partial v}{\partial \eta_g} +\left(v-2h\right)\frac{\partial f}{\partial \eta_g}\right]
\exp\left(-f\right)\,ds_g=0.
\]

Let us now compute the integral
\[
II=-\int_{\partial M} \nabla_i\exp\left(-f\right)v^{ij}\eta_j\,ds_g
=-\int_{\partial M}\nabla_i \exp\left(-f\right)v_{\,\,n}^i\,ds_g.
\]
Under the previous conventions,
\[
II=\int_{\partial M}\nabla_{\alpha}f\exp\left(-f\right)v_{\,\,n}^{\alpha}\,ds_g\\
+\int_{\partial M}\nabla_{n}f\exp\left(-f\right)v_{\,\,n}^n\,ds_g.
\]
Using the fact that $\nabla_n f=\partial_n f=0$ on $\partial M$, we can compute
\[
\nabla^{\alpha}\nabla_n f\nabla_{\alpha}f=-H\left(\nabla^{\top}f,\nabla^{\top}f\right),
\]
where $H$ denotes the second fundamental form of the boundary.
Hence, using the definition of $v_{\alpha n}$, we get
\begin{equation*}
II=\int_{\partial M}
2H\left(\nabla^{\top}f,\nabla^{\top}f\right)\exp\left(-f\right)\,ds_g
=\int_{\partial M}
2k_g\left|\nabla^{\top}f\right|^2\exp\left(-f\right)\,ds_g,
\end{equation*}
and the formula is proved.
\end{proof}

Next, we consider Perelman's $\mathcal{W}$-functional, namely,
\[
\mathcal{W}\left(g,f,\tau\right)=\int_{M}\left[\tau\left(\left|\nabla f\right|^2+R_g\right)+f-2\right]
\left(4\pi \tau\right)^{-1}\exp\left(-f\right)\,dA_g.
\]
Under the unnormalized Ricci flow (\ref{flow}), and the evolution equations
\begin{equation}
\label{normalperelman}
\left\{
\begin{array}{l}
\frac{\partial f}{\partial t}=-\Delta_g f+\left|\nabla f\right|^2-R_g+\frac{1}{\tau}
\quad\mbox{in}\quad M\times\left(0,T\right)\\
\frac{d\tau}{dt}=-1 \quad\mbox{in}\quad \left(0,T\right)\\
\frac{\partial f}{\partial \eta_g}=0 \quad\mbox{on}\quad \partial M\times\left(0,T\right),
\end{array}
\right.
\end{equation}
we have the following formula, which shows a monotonicity property
as long as $k_g\geq 0$ and $k_g'=\psi'\leq 0$, for the functional $\mathcal{W}$. 
\begin{theorem}
\label{monotonicity2}
\begin{eqnarray*}
\frac{d}{dt}\mathcal{W}&=&\int_M 2\tau\left|R_{ij}+\nabla_i\nabla_j f-\frac{1}{2\tau}g_{ij}\right|^2\left(4\pi\tau\right)^{-1}
\exp\left(-f\right)\,dA_g\\
&&+\frac{1}{4\pi}
\left(\int_{\partial M}\left(k_gR_g-2k_g'+2k_g\left|\nabla^{\top}f\right|^2\right)\exp\left(-f\right)\,ds_g\right).
\end{eqnarray*}
\end{theorem}

\begin{proof} 
Using the fact that
\[
0=\int_{\partial M}\frac{\partial e^{-f}}{\partial \eta_g}\,ds_g
=\int_M \Delta_g e^{-f}\,dA_g=\int_M \left(\left|\nabla f\right|^2-\Delta_g f\right)e^{-f}\,dA_g,
\]
and recalling that under (\ref{normalperelman}), $\delta\left(\frac{1}{4\pi\tau}e^{-f}dA_g\right)=0$ (\cite[Eq.~12.3]{KleinerLott}),
 we can compute the contribution to the formula due
to the variation of the term
\[
\frac{1}{4\pi\tau}\int_M \left(f-2\right)\exp\left(-f\right)\,dA_g.
\]
From this, and the computations in the proof of Theorem \ref{monotonicity1}, the theorem easily follows.
\end{proof}

\section{Controlling the injectivity radius of a surface with boundary}
\label{Non-collapsing}

\subsection{An extension procedure}
Here we show an extension procedure for surfaces
with boundary of positive scalar curvature and convex boundary that allows us to control easily the maximum of the curvature
of the extension (compare with the results in \cite{Kronwith}). 

\begin{theorem}
\label{extension}
Let $\left(M,g\right)$ be a compact surface with boundary, and assume that its Gaussian curvature and
the geodesic curvature of its boundary are strictly positive. Let $z_0>0$ be arbitrary. Then there exists a closed
surface $\left(\hat{M},\hat{g}\right)$, $\hat{g}$ a $C^2$ metric, such that $M$ is isometrically embedded in $\hat{M}$,
and the Gaussian curvature $\hat{K}$ of $\hat{M}$ is strictly positive and satisfies
\[
0<\hat{K}\leq K_{+}+\frac{2\alpha_+}{z_0},
\]
where $K_+$ is the maximum of the Gaussian curvature of $M$, and $\alpha_+$ is the maximum
of the geodesic curvature of $\partial M$. 
\end{theorem}
\begin{proof}
Let $\theta\in \partial M$.
Given $K\left(\theta\right)>0$ the (Gaussian) curvature function of
$M$ restricted to $\partial M$, define the following family of functions. First for $y<0$:
\[
K_{y}\left(\theta,\zeta\right)=
\left\{
\begin{array}{l}
K\left(\theta\right)+\frac{y K\left(\theta\right)}{z_0}\zeta \quad \mbox{if}\quad 0\leq\zeta<\frac{z_0}{1-y}\\
\frac{K\left(\theta\right)}{1-y} \quad\mbox{if}\quad \frac{z_0}{1-y}\leq\zeta\leq z_0 ,
\end{array}
\right.
\]
and for $y\geq 0$:
\[
K_{y}\left(\theta,\zeta\right)=K\left(\theta\right)+y\zeta,\quad 0\leq\zeta\leq z_0.
\]

Observe that for a given $\alpha>0$ there exists exactly one member of the previously
defined family, say $K_{y\left(\alpha\right)}$, such that
\begin{equation}
\label{fixedpoint}
\alpha\left(\theta\right)= \int_0^{z_0}K_{y\left(\alpha\right)}\left(\theta,\zeta\right)\,d\zeta.
\end{equation}
Indeed, notice that
 for fixed $\theta$ we have that 
\begin{equation*}
K_{y_1}\left(\theta,\zeta\right)< K_{y_2}\left(\theta,\zeta\right), \quad \zeta>0,\quad \mbox{whenever}\quad y_1<y_2,
\end{equation*}
\[
\int_0^{z_0}K_{y}\left(\theta,\zeta\right)\,d\zeta\rightarrow 0 \quad
\mbox{as}\quad y\rightarrow -\infty,
\]
\[
\int_0^{z_0}K_{y}\left(\theta,\zeta\right)\,d\zeta\rightarrow \infty \quad
\mbox{as}\quad y\rightarrow \infty,
\]
and also, for $\theta$ fixed, if $\alpha'\leq\alpha$, the corresponding functions
$K_{y\left(\alpha'\right)}$ and $K_{y\left(\alpha\right)}$ (which satisfy (\ref{fixedpoint}) for
$\alpha'$ and $\alpha$ respectively) satisfy
\[
K_{y\left(\alpha'\right)}\left(\theta,\zeta\right)\leq K_{y\left(\alpha\right)}\left(\theta,\zeta\right).
\]

We are ready to extend the metric from a convex surface with boundary to a compact 
closed surface, keeping control over the maximum of the curvature. Define the warping function
\[
f\left(\theta,z\right)= 1+\alpha\left(\theta\right)z-
\int_0^{z}\int_0^{\zeta}K_{y\left(\alpha\left(\theta\right)\right)}\left(\theta,\xi\right)\,d\xi\,d\zeta,
\]
where $\alpha\left(\theta\right)$ is the geodesic curvature of $\partial M$ at the point
$\theta\in \partial M$.
Notice that $z_0>0$ can be chosen arbitrarily, and also that $\displaystyle \frac{\partial f}{\partial z}\geq 0$ 
on $0\leq z\leq z_0$ and hence
$f\geq 1$  on the same interval. 

If $g_{\partial M}$ is the metric of $M$ restricted to its boundary, we define a
metric $\hat{g}$ on $N=\partial M\times \left[0,z_0\right]$ by
\[
\hat{g}=dz^2+f^2 g_{\partial M}.
\]
This metric defines an extension of the metric
on the surface $M$ to the surface $\hat{M}_0=M\cup N$ where $\partial M\subset M$ is
identified with $\partial M\times \left\{0\right\}\subset N$. It is clear that this
metric is $C^2$, that
$\partial \hat{M}_0= M\times\left\{z_0\right\}$ and that it is totally geodesic.

Let us now 
estimate the maximum of the curvature in our extension. 
Define
\[
K_+\left(\zeta\right)= K_++\dfrac{2\alpha_+}{z_0^2}\zeta,
\]
i.e., take from the family of functions defined above, in the case when $y\geq 0$, $K\left(\theta\right)=K_+$
and
$y=\dfrac{2\alpha_+}{z_0^2}$.
Observe that
\[
\int_0^{z_0}K_+\left(\zeta\right)\,d\zeta=K_+z_0+\alpha_+>\alpha_+.
\]
Therefore, by the properties discussed above for the family of functions $K_y$, we have 
\begin{eqnarray*}
K_{y\left(\alpha\left(\theta\right)\right)}\left(\theta,\zeta\right)&\leq& K_++\dfrac{2\alpha_+}{z_0^2}\zeta.
\end{eqnarray*}
Notice now that $K_{y\left(\alpha\left(\theta\right)\right)}=-\dfrac{\partial^2}{\partial z^2}f\left(z\right)$
and hence, by taking $\zeta=z_0$, we obtain
\[
-\frac{\partial^2}{\partial z^2}f\left(z\right)\leq K_{+} + \frac{2\alpha_+}{z_0}.
\]
Since $f\geq 1$, it follows that the Gaussian curvature of $\hat{M}_0$, which is equal
to $-f_{zz}/f$, is at most 
$K_++\frac{2\alpha_+}{z_0}$.
Given the fact that the produced extension $\hat{M}_0$ is a surface with a $C^2$ metric and
with a totally geodesic boundary, so we can double it
to obtain a closed surface endowed with a $C^2$ metric of positive Gaussian curvature, which is bounded above
by $K_++\frac{2\alpha_+}{z_0}$. 
\end{proof}

From the proof of the previous theorem we can extract the following useful corollary.
\begin{corollary}
\label{lengthshortest}
If there is a geodesic in $M$ of length $l$ that hits the boundary orthogonally at both its endpoints, then
there is a closed geodesic (which is $C^3$) in the extension $\hat{M}$ of length $2l+2z_0$. 
\end{corollary}

\subsection{}
Let $\left(M,g\right)$ be a compact surface with boundary. We will assume that its scalar curvature is
positive as well as the geodesic curvature of its boundary. We will assume that the bounds
$0<R\leq 2K_+$ and $0\leq k_g\leq \alpha_+$ hold, and also, without loss of generality, that $K_+\geq 1$.

Before we state the main result of this section, let us review the concept of injectivity 
radius of a surface with boundary. We shall need the following definitions:
\begin{definition}
Let $M$ be a manifold with boundary and $p$ a point in its interior
($p\in M\setminus \partial M$). Define
$\iota_{\mbox{int}}\left(p\right)$ as the supremum of $r>0$ such that if
\[
\gamma: \left[0,t_{\gamma}\right]\longrightarrow M
\]
is a normal geodesic with $\gamma\left(0\right)=p$, then it is minimizing from $0$
to $\min\left\{t_{\gamma},r\right\}$, where $t_{\gamma}$ is the first time
that $\gamma$ intersects $\partial M$.

\vspace{.1in}
\noindent
We define the interior injectivity radius of $M$ as
\[
\iota_{\mbox{int}}:= \inf\left\{\iota_{\mbox{int}}\left(p\right):\, p\in M\setminus \partial M\right\}.
\]
\end{definition}

\begin{definition}
For a Riemannian manifold with boundary $M$, and $p\in \partial M$, define
$\iota_{\partial}(p)$
as the supremum  $r>0$ 
such that any minimizing geodesic $\gamma$ issuing from $p$ normally to
$\partial M$ uniquely minimizes distance to $\partial M$ up to distance r (i.e., $\gamma(0)=p$ 
and $\mbox{dist}(\gamma(r),\partial M)=r$).

\vspace{.1in}
\noindent 
Define $i_{\partial}(M)$ the boundary injectivity 
radius of $M$ (as opposed to the injectivity radius of the boundary) as
\[
\iota_{\partial}\left(M\right)=\inf \{ i_{\partial}(p):\quad p\in \partial M \}.
\]
\end{definition}

The \emph{injectivity radius} $\iota_M$ of the surface is defined as 
\[
\iota_M=\min\left\{\iota_{\partial},\iota_{\mbox{int}}\right\}.
\]
From the definition of the injectivity radius for a surface with boundary,  and the Klingenberg estimates for
the injectivity radius of a compact surface of positive curvature
(see Theorem 1.114 in \cite{Chow3} and \cite[\S 6]{Kodani}), one can conclude that in the case of a surface with boundary,
we have an estimate
from below for the injectivity radius $\iota_M$ of a surface with boundary given by
\[
\iota_M \geq \min\left\{\mbox{Foc}\left(\partial M\right), \frac{1}{2}l,\frac{c}{\sqrt{K_+}}\right\},
\]
where $\mbox{Foc}\left(\partial M\right)$ is the focal distance of $\partial M$, $l$ is the lenght of the shortest geodesic meeting $\partial M$
at its two endpoints at a right angle, and $c>0$ is a universal constant.
For the benefit of the reader, let us recall the definition of the \emph{focal distance} of $\partial M$:

\vspace{.1in}
Let $\nu$ be the normal bundle of $\partial M$, and denote by $\nu^{-}$ denote the bundle of inward pointing
normal vectors. Then we can define the exponential map
\[
\exp: \nu^{-}\longrightarrow M,
\]
as
\[
\exp\left(p\right)=\gamma_p\left(1\right),
\]
where $\gamma_p$ is a geodesic starting at $p$ whose velocity vector is normal to $\partial M$ and
points inwards. For a compact surface, there is an $\delta>0$ for which this map is well defined when
restricted to
normal vectors to $\partial M$ of length at most $\delta$, so we will
think of this map as defined over this subset of $\nu^{-}$, and which we will denote by $\nu^{-}\left(\delta\right)$. 
We say that $p\in M$ is a focal point of $\partial M$ if $p$ is a critical point of the exponential map.
The \emph{focal distance} is the minimal distance of a focal point of $\partial M$ to $\partial M$. 

\vspace{.1in}
Since it is also well known from comparison geometry that (see \cite[\S 6]{Kodani})
\[
\mbox{Foc}\left(\partial M\right)\geq \frac{1}{\sqrt{K_+}}\arctan\left(\frac{\sqrt{K_+}}{\alpha_+}\right)
\quad\left(\geq\frac{\pi}{2\sqrt{K_+}}\quad \mbox{if}\quad \alpha_+=0\right),
\]
our estimate on the injectivity radius reduces to
\[
\iota_M \geq \min\left\{\frac{1}{2}l,\frac{c}{\sqrt{K_+}}\right\},
\]
for a new constant $c$.

Our wish now is to show that 
along the Ricci flow (\ref{flow}), with initial and boundary data satisfying the requirements of Theorem \ref{maintheorem},
 on any finite interval $\left(0,T\right)$ of time where it is defined there is a constant $\kappa>0$,
which may depend on $T$ but which is otherwise independent of 
time, such that at any time $t\in\left(0,T\right)$
\begin{equation}
\label{injectivitycontrol}
\iota_{\left(M,g\left(t\right)\right)}\geq \frac{\kappa}{\sqrt{R_{\max}\left(t\right)}}
\quad
\mbox{where}\quad 
R_{\max}\left(t\right)=\max_{p\in M}R_g\left(p,t\right).
\end{equation}
The desired estimate is then a consequence of the following estimate, which is an analogue of Klingenberg's Lemma
for surfaces of positive scalar curvature and convex boundary.
\begin{proposition}
\label{noncollapsing}
Let $\left(M,g\right)$ be a 
compact surface with boundary. Assume that the scalar curvature of $M$ satisfies $0<R\leq 2K_+$,
$K_+\geq 1$,
and that the geodesic curvature of the boundary satisfies $0\leq k_g\leq \alpha_+$.
Let $l$ be the length of the shortest geodesic in $M$ whose both endpoints
are orthogonal to the boundary. There is a constant $\kappa:=\kappa\left(\alpha_+\right)>0$
such that 
$\displaystyle
l\geq \frac{\kappa}{\sqrt{K_+}}$.
\end{proposition}
\begin{proof}
For this proof we may assume $k_g>0$, as we can deal with the case $k_g\geq 0$ by considering, instead of $M$,
\[
M\left(\epsilon\right)=\left\{p\in M:\, \rho\left(p\right)\geq \epsilon\right\},
\]
where $\rho$ is the distance function to $\partial M$; for any $\epsilon>0$ small enough, $R>0$ and $k_g\geq 0$, it is not difficult to show that
$M\left(\epsilon\right)$ has boundary with strictly positive geodesic curvature. Indeed, recall that
the geodesic curvature $k_g$ of a level curve of $\rho$ (see \cite[Proposition 6.2]{Cortissoz1}) satisfies 
\[
\dfrac{\partial k_g}{\partial \rho}=\dfrac{R}{2}+k_g^2.
\]
The proposition will be proved in this case 
then, by noticing that
a geodesic hitting the boundary of $M\left(\epsilon\right)$ orthogonally at both of its endpoints is at least $2\epsilon$ shorter than a geodesic with the same property
in $M$. 

Now, let $l$ be the length of the shortest geodesic that hits the boundary of $M$ orthogonally.
Let $z_0=\frac{\alpha_+}{2C\sqrt{K_+}}$, $C>0$ a constant to be chosen, in Theorem \ref{extension} and Corollary \ref{lengthshortest}.
By Corollary \ref{lengthshortest} and Klingenberg's injectivity
radius estimate applied to $\hat{M}$, the extension of $M$ given
by Theorem \ref{extension}, we have that
\[
2l+\frac{\alpha_+}{C\sqrt{K_+}}\geq \frac{c'}{\sqrt{K_+ + C\sqrt{K_+}}},
\]
where $c'$ is a universal constant. Hence we have an estimate for $l$:
\[
l\geq \frac{c'}{2\sqrt{K_+ + C\sqrt{K_+}}}-\frac{\alpha_+}{2C\sqrt{K_+}}\geq \frac{c'}{2\sqrt{1+C}\sqrt{K_+}}-\frac{\alpha_+}{2C\sqrt{K_+}},
\]
and to get the last inequality we have used that $K_+\geq 1$. If $\alpha_+\leq \frac{c'}{4}$, we can choose $C=1$.
If $\alpha_+>\frac{c'}{4}$, choose $C>0$ so that
\[
\frac{c'}{2\sqrt{1+C}\sqrt{K_+}}-\frac{\alpha_+}{2C\sqrt{K_+}}\geq \frac{\alpha_+}{2C\sqrt{K_+}}, 
\]
by taking, for instance, $\displaystyle C= \frac{4\alpha_+^2+\sqrt{16\alpha_+^4+16\left(c'\alpha_+\right)^2}}{2\left(c'\right)^2}$. This shows the proposition.
\end{proof}

\subsection{Compactness for Ricci flows}

Before we state and prove any theorem regarding compactness properties of the Ricci flow, 
let us review the definition of a Fermi chart.
Given $p\in \partial M$ pick normal coordinates around $p$ in $\partial M$ (with respect
to the metric restricted to the boundary) $\left(B_r\left(0\right),\psi\right)$
\[
\psi:B_{r}\left(0\right)\longrightarrow \partial M, \quad \psi\left(0\right)=p,
\]
$B_r\left(0\right)$ being a ball in euclidean space centered at $0$ and of radius $r$.
Now define a map
\[
\varphi: U=B_r\left(0\right)\times \left[0,\delta\right)\longrightarrow  M
\]
as follows
\[
\varphi\left(x,s\right)=\gamma_{\psi\left(x\right),e_2}\left(s\right),
\]
where $e_2$ is the inward pointing normal at $\psi\left(x\right)$, and 
$\gamma_{x,e_2}$ is the geodesic departing from $\psi\left(x\right)$ and velocity $e_2$.
For $\delta>0$ small enough
$\varphi$ is a homeomorphism onto its image. We call $\left(U,\varphi\right)$ a Fermi chart,
and we refer to $\left(r,\delta\right)$ as its size.

We shall sketch a proof of the following result: 
\begin{theorem}
\label{compactnessthm}
Let $\left(M_k,O_k,G_k\right)$, $G_k=e^{u_k}g_{0,k}$, be a sequence of solutions to the Ricci flow
defined on time interval $\left(A,\Omega\right)$ (we shall assume
that $0\in\left(A,B\right)$). Suppose that

(i) The absolute values of $u_k$ and the sectional curvatures of the $M_k$ 
are bounded independently of $k$ at all times $A<t<\Omega$;

(ii) The geodesic curvature $k_{g_k}$ of $\partial M_k$ and its covariant derivatives with respect
to the metric $g_{0,k}$ are bounded indepentenly of $k$
; 

(iii) There exists a $\delta>0$ such that the injectivity radii, $\iota_k$, of $M_k$, and
the injectivity radii, $\iota_{b,\kappa}$, 
of the boundaries $\partial M_k$ (with respect to the induced metrics), 
at $t=0$ satisfies
\[
\iota_{k}\geq \delta, \quad \iota_{b,\kappa}\geq \delta;
\]

(iv) There exists  a $\rho>0$ such that 
given any point $q_k\in M_k$, there is a Fermi chart $\left(U,\varphi\right)$ around $q_k$ with respect to
$g_{0,k}$ of size $\left(\rho,\rho\right)$, where the following holds:

(a) There is a $\Lambda>0$ such that in the chart
\[
\Lambda^{-2}\delta_{ij}\leq \left(g_{0,k}\right)_{ij}\leq \Lambda^2\delta_{ij},
\]

(b) for each integer $l>0$ there is a $C_l>0$ independent of $k$ such that
\[
\sup_{x\in U}\left|\partial^{\beta} \left(g_{0,k}\right)_{ij}\left(x\right)\right|\leq C_l, \quad \mbox{for}
\left|\beta\right|=l.
\]
where the symbol $\partial$ represents partial differentiation in the chart.

Then there exists a subsequence which converges to a solution to the Ricci flow. 

\end{theorem}

As pointed out in \cite{Hamilton},
it suffices to prve the previous theorem in the case when
$-\infty<A<\Omega<\infty$.
Another interesting observation is that hypothesis (i) in Theorem \ref{compactnessthm} can be weakened:
we only need uniform bounds on $u_k\left(\cdot,t_k\right)$ for
a sequence of times $t_k\in\left(A,B\right)$, as long as we have a bound on the curvature, since
\[
u_k\left(x,t\right)=u_k\left(x,t_k\right)-\int_{t_k}^tR\left(x,\tau\right)\,d\tau.
\]

Following the ideas from \cite{Hamilton}, the first step to take in the prove of the previous theorem is to show that 
the sequence $\left(M_k,O_k,G_k\left(0\right)\right)$ has 
a convergent subsequence in the Gromov-Hausdorff sense. 
Let $\left(U,\varphi\right)$ be a chart. In what follows we will use the notation
\[
\left\|g\right\|_{C^{k}\left(U,\varphi\right)}=\sup_{x\in U}\sum_{\left|\beta\right|\leq k}\left|\partial^{\beta}g_{ij}\left(x\right)\right|
\]
where the $\partial$ indicates partial diferentiation in the coordinates given by the chart.

The following theorem will be useful in this endeavour.

\begin{theorem}
\label{compactnessB}
Let $\left\{\left(M_k,O_k,g_{k}\right)\right\}_{k\in\mathbb{N}}$, a sequence of pointed Riemannian manifolds with boundary.
Assume that the following holds:

(i) For all $p\in\partial M_k$ there exists a Fermi chart $\left(U,\varphi\right)$ of size $\left(\epsilon,\delta\right)$  (the size independent
of $p$ and $k$), and so that for every integer $l>0$ there exists a $C_l$ which only depends on $\left(\epsilon,\delta\right)$ such that
\[
\left\|\left(g_k\right)_{ij}\right\|_{C^l\left(U,\varphi\right)}\leq C_l,
\]

(ii) For all points $p\in M_k\setminus \partial M_k$, there is a normal coordinate chart of radius $\rho\left(\eta\right)>0$, where
$\eta$ is the distance from $p$ to $\partial M_k$ with respect to the metric $g_k$,
in which for every integer $l>0$ we have
\[
\left|\nabla^l Rm_k\right|_k<C_l,
\]
with $C_l$ independent of $k$ but which may depend on $\rho$ (here $\left|\cdot\right|_k$ represents norm 
with respect to the metric $g_k$).

Then, for any $m>0$ there is a subsequence converging in the $C^{m}$ Hausdorff-Gromov topology to either
a manifold with boundary or a manifold without boundary. 
\end{theorem}

\subsection{Proof of Theorem \ref{compactnessB}}
Theorem \ref{compactnessB} is a consequence of Theorem \ref{fundamentalcompactness} in the appendix.
Indeed, from the hypotheses of Theorem \ref{compactnessB} we have the following.

\begin{lemma}
Under the hypothesis of Theorem \ref{compactnessB} we have the following:

Let $m>0$ and $Q\in\left(1,2\right)$. Then there exists a $\rho>0$ such that any point has 
either a Fermi chart of size at least $\left(\rho,\rho\right)$ or a normal coordinate
chart of radius at least $\rho$ so that in either chart holds:
\begin{enumerate}
\item
\[
Q^{-2}\delta_{ij}\leq g_{ij}\leq Q^2\delta_{ij},\quad \mbox{and}
\]
\item
\[
\left|\partial^{m+1}g_{ij}\right|\leq Q-1.
\]
\end{enumerate}
\end{lemma}
\begin{proof}
First we treat the case of $p\in \partial M$. By the hypotheses of the lemma, we can find a Fermi chart of 
size $\left(\epsilon,\epsilon\right)$ (this size independent of $p$), $\epsilon>0$, and $C_{m+2}$ such that
\[
\left|\partial^{m+2}g_{ij}\right|\leq C_{m+2}.
\]
But then we can make $\rho<1$ a bit smaller so that (2) holds. Also, using (i) in Theorem
\ref{compactnessB} with $m=0$, since $g_{ij}(x,0)=\delta_{ij}$ in the Fermi chart, we can make
(1) hold, again by reducing $\epsilon$, and noticing that how much we have to reduce it
only depends on $C_1$.

If $p\in M$, and is not covered by a Fermi chart, then the distance to the boundary 
is at least $\rho$, so 
 we can work with normal coordinate charts of radius
$\leq \dfrac{\rho}{2}$ where we can control the metric and its derivatives 
(in the chart) in terms of bounds on the curvature and its covariant derivatives (see Theorem 4.11 in \cite{Hamilton}),
bounds which can be given in terms of bounds on the curvature (which is assumed), via Shi's local interior derivative estimates
(Theorem 6.9 in \cite{Chow3}). 
\end{proof}

\subsection{Proof of Theorem \ref{compactnessthm}}
First we have the following result.
\begin{lemma}
\label{controllingderivatives}
Assume that we have a sequence of Ricci flows $\left(M,O_k,e^{u_k}g_{0,k}\right)$, on $t\in \left(A,B\right)$. Assume
that $u_k$ is uniformly bounded and  assume that there is
a $\epsilon,\delta>0$ such that every $p\in M_k$ has a Fermi chart with respect to $g_{0,k}$ of size $\left(\epsilon,\delta\right)$. 
If for all integer $l>0$ and every Fermi coordinate
chart $\left(U,\varphi\right)$ with respect to $g_{0,k}$ of size $\left(\epsilon,\delta\right)$ we have that 
\[
\left\|g_{0,k}\right\|_{C^l\left(U,\varphi\right)}\leq N_l,
\]
and
\[
\alpha\delta_{ij}\leq \left(g_{0,k}\right)_{ij}\leq \beta\delta_{ij},
\]
then we have uniform control (independent of $k$) over the metric $g_k=e^{u_k}g_{0,k}$ 
and its derivatives (with respect to the chart) on any compact subset of $U\times\left(A,B\right)$. Namely,
For any integer $l>0$, $U'\subset U$  
such that $\overline{U'}\subset U$, and $\left[A',B'\right]\subset \left(A,B\right)$ there exists a $C_l$
which may depend on $U',A',B',\alpha,\beta$ and $N_m$
such that
\[
\left\|g_{k}\left(t\right)\right\|_{C^l\left(U',\varphi\right)}\leq C_l,\quad t\in\left[A',B'\right].
\]
\end{lemma}
\begin{proof}
This is an immediate consequence of Lemma \ref{derivativesconformalfactor} in the appendix and Shi's local interior
derivative estimates.  
\end{proof}
\begin{lemma}
\label{Fermicoordinates}
Let $\left(x,s\right)$ be Fermi coordinates, let $k_g\left(\cdot,s\right)$
be the geodesic curvature of the curve at distance $s$ from the boundary. Then we have:
\begin{enumerate}
\item 
\[
g=ds^2+\exp\left(-2\int_0^sk_g\left(x,\sigma\right)\,d\sigma\right)dx^2,
\]
\item
\[
\dfrac{\partial k_g}{\partial s}=\dfrac{R}{2}+k_g^2,
\]
\item
\[
\sqrt{\left|g\right|}=\exp\left(-2\int_0^sk_g\left(x,\sigma\right)\,d\sigma\right).
\]
\end{enumerate}
\end{lemma}
\begin{proof}
See the proof of Proposition 6.2 in \cite{Cortissoz1}.
\end{proof}

\begin{lemma}
\label{fromcovariant}
Let $\left(M,g\right)$ a Riemannian surface with boundary.
Assume we have bounds on the covariant derivatives of the curvature and on the covariant derivatives
of the geodesic curvature, i.e.,
\[
\left|\nabla^m R\right|_g\leq C_m \quad \mbox{and}\quad \left|\overline{\nabla}^mk_g\right|_g\leq C_m,
\]
where $\overline{\nabla}$ indicates covariant differentiation in $\partial M$ with respect to the induced
metric, and $\left|\cdot\right|_g$ denotes norm with respect to $g$.
Then, in a Fermi chart of size $\left(\epsilon,\delta\right)$
there are $K_l$ which may depend on $\epsilon,\delta$ and $C_m$ such that
\[
\left\|R\right\|_{C^{l}\left(U,\varphi\right)}\leq K_l.
\]
\end{lemma}
\begin{proof}
from the fact that $g_{ij}\left(x,0\right)=\delta_{ij}$ and equation (2) in Lemma \ref{Fermicoordinates} we get that there exists $Q$ which depends
on bounds on the curvature, on the geodesic curvature and the size of the Fermi chart
such that
\[
Q^{-2}\delta_{ij}\leq g_{ij}\leq Q^2\delta_{ij}.
\]
Now notice that control over the gradient of the curvature and the previous estimate allow us to control the
derivatives of the curvature in the chart. 
Indeed,
\[
\left|\partial_k R\right|= \left|g_{jk}\nabla^j R\right|\leq Q^{2}\left|\nabla R\right|.
\]
On the other hand we can also obtain bounds on $\partial_x k_g\left(x,s\right)$.
Indeed,
we have an equation
\[
\partial_s\partial_x k_g=\dfrac{1}{2}\partial_x R +2k_g\partial_x k_g
\] 
with initial condition $\partial_x k_g\left(x,0\right)$ (a bound on which we are assuming: covariant derivatives of the geodesic curvature 
of the boundary coincide in this chart with usual derivatives). 
By Lemma \ref{Fermicoordinates} we can write
\[
g_{11}=1,\quad g_{22}=\exp\left(\int_0^s -2k\left(x,\sigma\right)\,d\sigma\right).
\]
Hence,
we have that 
\[
\dfrac{dg_{22}}{ds}=-2k\left(x,s\right)\exp\left(\int_0^s -2k\left(x,\sigma\right)\,d\sigma\right),
\]
and
\[
\dfrac{dg_{22}}{dx}=-2\int \partial_x k\left(x,s\right)\,d\sigma\exp\left(\int_0^s -2k\left(x,\sigma\right)\,d\sigma\right),
\]
But then we have control over the Christoffel symbols, and this gives us control over second derivatives of the curvature
from control over the second covariant derivatives (the Hessian)
of the curvature. From this control via Lemma \ref{Fermicoordinates}
(by taking derivatives of the formulas) we get control over second derivatives
of the metric; this gives us control over the derivatives of the Christoffel symbols, and hence over
third derivatives of the curvature from bounds on third covariant derivatives of the metric. 
Proceeding inductively we obtain the result.
\end{proof}

We are ready to start with a proof of Theorem \ref{compactnessthm}. 
First we will show that $\left(M_k,O_k,G_k\left(0\right)\right)$, and for that
we will show that this sequence of pointed surfaces satisfies the
assumptions of Theorem \ref{compactnessB}:

By (i) and (iv) in Theorem \ref{compactnessthm}, the hypotheses of Lemma 
\ref{controllingderivatives} hold, and hence for each $k$ 
there is a family of Fermi charts of uniform size say $(\epsilon,\delta)$ with respect to
$g_{0,k}$, this size independent of $k$, and
which cover $\partial M_k$, where 
we have uniform control over the curvature and its derivatives on any compact interval of time,
and hence on its covariant derivatives:
Lemma \ref{controllingderivatives}  also gives uniform control over the Christoffel symbols
and its derivatives in the Fermi charts.

As we have uniform control over the supremums of  $\left|u_k\right|$ (as assumed in (i)),
the arguments in the previous paragraph imply that there exists an $\eta>0$ so that we can bound the curvature and its covariant 
derivatives, independently of $k$, in the set
\[
\left\{p\in M_k:\, \mbox{dist}_{G_k\left(0\right)}\left(p,\partial M_k\right)\leq \eta\right\},
\]
where the distance is measured with respect to $G_k\left(0\right)$.
By (iii) (in Theorem \ref{compactnessthm}), at $t=0$, there is a $\rho>0$  (which we may
assume $\leq \eta$) so that any $p\in M_k$ has
a Fermi chart of size $\left(\rho,\rho\right)$ with respect to $G_k\left(0\right)$.
Therefore,  by Lemma \ref{fromcovariant}, for any $k$, 
at time $t=0$ we have that given $p\in \partial M_k$ there is a Fermi chart, whose size is independent of $k$ and $p$,
where we have control over the curvature and its derivatives. Hence, using Lemma \ref{Fermicoordinates},
we obtain control over the metric and its derivatives in this Fermi chart, so (i) 
for $\left(M_k,O_k,G_k\left(0\right)\right)$ in Theorem \ref{compactnessB} holds.

On the other hand, given a point whose distance from the boundary is at least $\rho/2$ (with respect to $g_k\left(0\right)$),
applying the local derivative estimates gives us a 
 normal coordinate chart (whose size is independent of $k$, again by the injectivity radius bound), 
 satisfying (ii) in Theorem \ref{compactnessB}. Indeed, take $\xi>0$ small, and let $U_k\subset M_k$
be a ball of radius $r>0$ with respect to $G_k\left(0\right)$ (the minimum between $\rho/2$ and the
injectivity radius). Then since we have control on the curvature, at time $t=-\xi$,
$U_k$ contains a ball of radius $r e^{-K\xi}$ (with respect to the metric $G_k\left(-\eta\right)$), where $K$ 
is a bound on the scalar curvature $R$, and hence by the local derivative estimates
(Theorem 6.9 in \cite{Chow3}),
at time $t=0$ we have control over  the norm, with respect to $G_k\left(0\right)$, of all the covariant derivatives of the curvature in $U_k$,
and these bounds are independent of $k$. 
Therefore at $t=0$, by Theorem \ref{compactnessB}, the sequence $\left(M_k,O_k,G_k\left(0\right)\right)$ has a convergent subsequence, say
to a manifold $\left(M,O,G\left(0\right)\right)$. 

Using again Shi's local derivative estimates and Lemma 
\ref{derivativesconformalfactor}, we have uniform bounds on $G_k$ and all its derivatives on any compact subset of $M_k\times \left(A,\Omega\right)$.
Then proceeding as in \cite[\S 2]{Hamilton} (in particular see the argument starting in the paragraph right before
Lemma 2.4) we can find a subsequence of the original sequence of Ricci flows which converges on 
$M\times\left(A,\Omega\right)$ towards a solution to the Ricci flow.

\subsection{Application to blow-up limits of the Ricci flow}
\label{blowuplimits}

Recall that 
a blow up limit is constructed as follows: if $\left(0,T\right)$, $0<T<\infty$, is the maximal
interval of existence for a solution to (\ref{flow}), we pick a sequence of
times $t_j\rightarrow T$ and a sequence of points such that
\[
\lambda_j:=R_g\left(p_j,t_j\right)=\max_{M\times\left[0,t_j\right]}R_g\left(x,t\right),
\]
and then we define the dilations
\[
g_j\left(t\right):=\lambda_j g\left(t_j+\frac{t}{\lambda_j}\right), \quad 
-\lambda_j t_j<t<\lambda_j\left(T-t_j\right).
\]
So assume that we have a solution to (\ref{flow}), with convex boundary, we will show that we can construct 
blow-up limits, via an application of Theorem \ref{compactnessthm}.

In order to be able to apply Theorem \ref{compactnessthm}, we must be able to control the derivatives of the
metric in certain Fermi charts (assumption (iv) in Theorem \ref{compactnessthm}). 
Since a solution to (\ref{flow}) can be written as $g=e^ug_0$, with $u\left(x,0\right)=1$,
in the case of a blow-up sequence, by the previous discussion, we can write $g_j=e^{u_j}g_{0,j}$, 
where $g_{0,j}=\lambda_j g_0$ ($\lambda_j\rightarrow \infty$) and
$u_j\left(t\right)=u\left(t_j+\frac{t}{\lambda_j}\right)$.
Notice then that at time  $t'_j=-\lambda_jt_j$, $u_j\left(x,t'_j\right)=0$.
We can get the control needed in Lemma \ref{controllingderivatives},
as the following two lemmas show.
\begin{lemma}
Let $\left(U,\varphi\right)$,
$U=\left(-\epsilon,\epsilon\right)\times\left[0,\delta\right)$, a Fermi chart for $g_1$. Then if $g_2=\lambda g_1$,
the chart $\left(U',\varphi'\right)$ given by
\[
U'=\left(-\sqrt{\lambda}\epsilon,\sqrt{\lambda}\epsilon\right)\times\left[0,\sqrt{\lambda}\delta\right),
\]
\[
\varphi'\left(x',s'\right)=\varphi\left(\dfrac{x'}{\sqrt{\lambda}},\dfrac{s'}{\sqrt{\lambda}}\right),
\]
is a Fermi chart for $g_2$. 
\end{lemma}
This previous lemma should be compared with the scaling properties of the harmonic radius (see \cite{Anderson}).

As a consequence, via the chain rule, we have that: 
\begin{lemma}
Let $\left(U,\varphi\right)$, $U=\left(-\epsilon,\epsilon\right)\times\left[0,\delta\right)$, be the
corresponding Fermi chart for $g_0$, and
$\left(U',\varphi'\right)$ be a Fermi chart for $g_2=\lambda g_1$. Then, if 
\[
c\delta_{ij}\leq \left(g_1\right)_{ij}\leq C\delta_{ij}
\]
holds for $U$, in $U'$ holds
\[
c\delta_{ij}\leq \left(g_2\right)_{ij}\leq C\delta_{ij}.
\]
Also, we have that
\[
\left\|\partial^l \left(g_2\right)_{ij}\right\|_{C^0\left(U',\varphi'\right)}=\dfrac{1}{\lambda^{\frac{l}{2}}}\left\|\partial^l \left(g_1\right)_{ij}\right\|_{C^0\left(U,\varphi\right)},
\]
and for the geodesic curvature
\[
\left\|\partial^l k_{g_2}\right\|_{C^0\left(\left(-\sqrt{\lambda}\epsilon,\sqrt{\lambda}\epsilon\right)\times\left\{0\right\},\varphi'\right)}
=\dfrac{1}{\lambda^{\frac{l+1}{2}}}\left\|\partial^l k_{g_1}\right\|_{C^0\left(\left(-\epsilon,\epsilon\right)\times\left\{0\right\},\varphi\right)}.
\]
\end{lemma} 

Also, in a blow-up sequence $u_k$ is uniformly bounded on any bounded interval of time, since as
we said before
there is a $t_k$ where $u_k\left(x,t_k\right)=0$, and as we have uniform bounds on the
curvature, we have a uniform bound on $u_k$.
The needed control over the injectivity radius of the surface and its boundary are provided by 
the results in Section \ref{noncollapsing}. On the other hand when the boundary is convex, 
a bound from below on $\iota_M$ gives a bound from below on the length of the boundary 
(and hence on its injectivity radius, since in this case $\partial M$ is a curve), therefore,
there exists $\epsilon>0$ and $\delta>0$ so that around any point $p\in\partial M$ we have a 
Fermi chart of size $\left(\epsilon,\delta\right)$ around $p$.
This shows that all the assumptions of Theorem \ref{compactnessthm} hold, and hence
we can construct blow-up limits.

\section{Proof of Theorem \ref{maintheorem}}
\label{smoothconvergence}

As shown in the previous section, given $\left(M,g\left(t\right)\right)$ a solution to (\ref{flow}) in a maximal time
interval $0<t<T<\infty$ of positive scalar curvature and convex boundary ($k_g\geq 0$) we can produce blow up limits. 

In our case, we can classify the possible blow up limits we may obtain.
We have the following result which, with minor modifications, is essentially proved in \cite{Cortissoz1}.
\begin{proposition}
\label{blowuplimits}
Let $\left(M,g\left(t\right)\right)$, $M$ a compact surface with boundary, be a solution to (\ref{flow}).
Let $\left(0,T\right)$, $T<\infty$, 
be the maximal interval of existence of $g\left(t\right)$. Assume that there is an $\epsilon>0$ such that for
all $0<t<T$, $R_{g}>-\epsilon$, and that $k_g$ is bounded. 
There are two possible blow up limits for $\left(M,g\left(t\right)\right)$ as $t\rightarrow T$.
If the blow up limit is compact, then it is a homotetically shrinking round hemisphere with
totally geodesic boundary. If the blow up limit is non compact then it is (or its double is) a cigar
soliton. 
\end{proposition}
\begin{proof} Just notice that any blow up limit of $\left(M,g\left(t\right)\right)$ as $t\rightarrow T$, will have nonnegative scalar curvature, 
which is strictly positive at one point,
a totally geodesic boundary, or no boundary at all, and will be defined in an interval of time $\left(-\infty, \Omega\right)$ (and $\Omega$ could be $\infty$). Then, by
doubling the manifold if needed (which as discussed in Appendix \ref{regularityricciflow}, gives a smooth solution to the Ricci flow), 
everything reduces to the boundaryless case, and \cite[Thm. ~26.1~and ~Thm. ~26.3 ]{Hamilton2} can be applied to give the 
proposition.
\end{proof}
Now we proceed with the
proof of Theorem \ref{maintheorem}. In what follows, we let $\left(M,g_0\right)$ be a compact surface with boundary of positive scalar curvature, and such that
$\partial M$ has nonnegative geodesic curvature, $\psi$ be as described in the statement of Theorem
\ref{maintheorem}, and we let $\left(M,g\left(t\right)\right)$ be the solution to (\ref{flow}) associated to
the initial data $g_0$ and the boundary data $\psi$. 
By Proposition \ref{blowupprop} we have $R>0$
and $R$ blows up in finite time, say $T$, and as discussed in the last paragraph of the previous section,
 we are allowed to take blow up limits of 
$\left(M,g\left(t\right)\right)$ as $t\rightarrow T$.
Hence, from Proposition \ref{blowuplimits}, and the monotonicity formula in \S \ref{Perelmanmonotonicity} (Theorem \ref{monotonicity2}), 
it follows that along a sequence of times
$t_k\rightarrow T$, $T<\infty$ being the maximum time of existence of $g\left(t\right)$, it holds that
\begin{equation}
\label{uniformlimit}
\lim_{k\rightarrow\infty}\frac{R_{\max}\left(t_k\right)}{R_{\min}\left(t_k\right)}=1,
\end{equation}
where, obviously,
\[
R_{\max}\left(t\right)=\max_{p\in M}R_g\left(p,t\right)
\quad \mbox{and} \quad R_{\min}\left(p,t\right)=\min_{p\in M}R_g\left(p,t\right).
\]
Indeed, as is the case with closed surfaces, the monotonicity
formula provided by Theorem \ref{monotonicity2} (as long as $\psi\geq 0$ and $\psi'\leq 0$) precludes the cigar as a blow up limit 
(see \cite[\S\S ~7.1]{Cortissoz1} and \cite[Corollary ~D.48]{Chow2}).
To see this more clearly, double the blow-up limit: it is an ancient solution to the Ricci flow 
(now without boundary), and hence it is smooth (see the discussion in Appendix \ref{regularityricciflow}); so we have two possible scenarios.
The obtained ancient solution is a Type I solution, in which case it must be a sphere; or it is a Type II solution. In 
this case, as the curvature of the solution assumes
its maximum at an origin, it must be the cigar. But then, this origin must located at the boundary of the
original blow up limit (before doubling) for
otherwise the curvature would assume two maximums, which does not happen in the cigar;
so this blow-up type II limit is half a cigar (as cut along a radial geodesic) if it has a boundary,
or just a cigar (if the limit is boundaryless). Hence, we can apply the arguments
in \cite[\S~7.1, pp. 45-46]{Cortissoz1} and define a family of 
functions $\phi$ in half the cigar (restricting the definition of $\phi$ given
for the whole cigar in the obvious way) or the whole cigar (if the limit is boundaryless) 
 to show that the functional $\mathcal{W}$ has as its infimum $-\infty$, a contradiction
with the fact that in the original Ricci flow (the one we extracted the limit from) cannot have this infimum
to be $-\infty$ as $\mathcal{W}$ is monotone along the Ricci flow (Theorem \ref{monotonicity2}), and at time $t=0$ 
this infimum is finite. 
Therefore
the only possible blow up limit is the round hemisphere with totally geodesic boundary, which implies (\ref{uniformlimit}).

The following interesting estimate on the evolution of the area $A\left(t\right):=A_{g\left(t\right)}\left(M\right)$ of $M$
under the Ricci flow (\ref{flow}), with initial condition $g_0$, can now be proved. 
\begin{proposition}
\label{areaestimate}
There exists constants $c_1,c_2>0$ such that
\[
c_1\left(T-t\right)\leq A\left(t\right)\leq c_2\left(T-t\right).
\]
\end{proposition}
\begin{proof}
Since in our case any blow up limit is compact, we must have 
\[
\lim_{t\rightarrow T}A\left(t\right)=0.
\]
Since $R>0$, and $\int_{\partial M}k_g \,ds_g$ is nonincreasing, by the Gauss-Bonnet Theorem we have
the inequalities
\[
-2\pi \leq \frac{dA}{dt}\leq -c,
\]
and the result follows by integration.
\end{proof}

As a consequence of the previous proposition and from the normalization (\ref{normalization}) we can immediatly conclude the following. 
\begin{corollary}
The normalized flow exists for all time.
\end{corollary}
\begin{proof}
The normalized flow exists up to time
\[
\lim_{t\rightarrow T^-}\int_{0}^{t}\frac{1}{A\left(\tau\right)}\,d \tau=\infty,
\]
by Proposition \ref{areaestimate}.
\end{proof}

Also, an estimate for the maximum of the scalar curvature can be deduced.
\begin{proposition}
\label{curvatureestimate}
There are constants $c_1,c_2>0$ such that
\[
\frac{c_1}{T-t}\leq R_{\max}\left(t\right)\leq \frac{c_2}{T-t}.
\]
\end{proposition}
\begin{proof}
By the Gauss-Bonnet theorem and the fact that, under the hypothesis of Theorem \ref{maintheorem},
$\displaystyle \int_{\partial M}k_g\,ds_g$ is nonincreasing, we have that
\begin{eqnarray*}
\int_{M}R_{\max}\left(t\right)\,dA_g \geq C,
\end{eqnarray*}
and from Proposition \ref{areaestimate}, there is a $c'>0$ such that 
\[
c' R_{\max}\left(t\right)\left(T-t\right)\geq C,
\]
so the left inequality follows.

To show the other inequality we proceed by contradiction. Assume that
there is no constant $c_2>0$ for which
\[
R_{\max}\left(t\right)\leq \frac{c_2}{T-t}
\]
holds. Then we can find a sequence of times $t_j\rightarrow T$
such that
\[
R_{\max}\left(t_j\right)\left(T-t_j\right)\rightarrow \infty,
\]
and hence along this sequence the blow up limit would not be compact, as 
it would have infinite area due to Proposition \ref{areaestimate}, and this would
contradict the arguments following Proposition \ref{blowuplimits}.
\end{proof}

Corollary \ref{curvatureestimate} shows that along any sequence of times we can take a blow up limit
since for any sequence of times, the curvature is blowing up at maximal rate 
(i.e. $\sim \frac{1}{T-t}$).
By the arguments following Proposition \ref{blowuplimits} this blow up limit is a round homotetically shrinking sphere. 
This proves the following theorem.
\begin{theorem}
Let $\left(M,g_0\right)$ be a compact surface with boundary, of positive
scalar curvature and such that the geodesic curvature of $\partial M$ is nonnegative, 
and let $\psi$ be as in the statement of Theorem \ref{maintheorem}.
Then, the solution to the Ricci flow (\ref{flow}) 
with initial condition $g_0$ blows up in finite time $T$, and for the scalar curvature $R$ we have that
\[
\lim_{t\rightarrow T}\frac{R_{\max}\left(t\right)}{R_{\min}\left(t\right)}=1.
\]
As a consequence, under the corresponding normalized flow we obtain that
\[
\tilde{R}_{\max}\left(\tilde{t}\right)-\tilde{R}_{\min}\left(\tilde{t}\right)\rightarrow 0 
\quad \mbox{as}\quad \tilde{t}\rightarrow \infty.
\] 
\end{theorem}

Since one can produce bounds on the derivatives of the curvature from bounds on the curvature,
and the curvature remains bounded along the normalized flow,
we have that along any sequence of times $\tilde{t}_n\rightarrow \infty$, there
is a subsequence of times $\tilde{t}_{n_k}\rightarrow\infty$ such that  $\tilde{g}\left(\tilde{t}_{n_k}\right)$ 
(recall that $\tilde{g}$ is the normalized flow corresponding to $g$) is converging to 
a metric of constant curvature  (these metrics may be different according to the sequence 
considered). Notice that the metrics $\tilde{g}\left(t\right)$ have its maximum 
curvature uniformed bounded from below, by Propositions \ref{areaestimate} and \ref{curvatureestimate}
(this precludes the flat cylinder as a possible limit), and hence
the curvature approaches a positive constant -along any sequence of times.
To be able to conclude that these limit metrics are isometric to that
of a standard hemisphere, we must also show that the geodesic curvature 
of the boundary approaches 0. Let us now conclude the proof of Theorem \ref{maintheorem}.

\medskip
\noindent
{\bf Finishing the proof of Theorem \ref{maintheorem}.}
 The only part of the statement that has not been proved in the previous discussion is that regarding 
the behavior of the geodesic curvature.
Notice that 
\[
\frac{c}{T-t}\leq \phi\left(t\right)\leq \frac{C}{T-t},
\]
where $\phi$ is the normalizing factor defined in the introduction, and hence
\[
T-t\leq Te^{-c\tilde{t}},
\]
which shows that
\[
k_{\tilde{g}}\leq C\sqrt{T-t}\psi\leq ce^{-c\tilde{t}}\psi.
\]
Since $\psi\geq 0$ and $\psi'\leq 0$, it remains bounded, and the theorem follows.
\hfill $\Box$

\medskip
\noindent

\section{Proof of Theorem \ref{maintheorem2}}
\label{neggeod}

In this section, we let $g\left(t\right)$ be the solution to (\ref{flow}) in the two-ball $D$, with initial and boundary data
as described in the hypotheses of Theorem \ref{maintheorem2}. It is clear, from uniqueness, that this solution
is also rotationally symmetric. Let $A\left(t\right):=A_{g\left(t\right)}\left(D\right)$ be the area of $D$ with respect to $g\left(t\right)$.
Given the fact that
\[
-c\leq\frac{d}{dt}A\leq -C<0
\]
to show that the normalized flow does exists for all time, all we must prove is that $A\left(t\right)\rightarrow 0$ as $t\rightarrow T$.

First notice that the scalar curvature of the solution to (\ref{flow}), with initial data $g_0$ and boundary data $\psi$ as in Theorem \ref{maintheorem2},
 blows up in finite time
by Proposition \ref{otherblowups}. Let us denote by $\mathcal{R}\left(t\right)$ the radius of $D$ with respect to $g\left(t\right)$; 
 by comparison geometry we have that at any time $\mathcal{R}\geq \frac{\pi}{2\sqrt{R_{\max}\left(t\right)}}$, and 
using Hamilton's arguments (see \cite[\S~5]{Chow0}), it can be shown that for points at distance at least $\frac{1}{4}\mathcal{R}$  from the
boundary, the injectivity radius is conveniently bounded from below.
To obtain Fermi charts of uniform size around any boundary point, notice that all we need is a bound
from below on the length of the boundary which conveniently rescales with the curvature: This is given 
by Corollary 1.2 (c) in \cite{AlexanderBishop0}

 Hence we can take a blow up limit of $\left(M,g\left(t\right)\right)$
as $t\rightarrow T$. This blow up
limit might be compact, and in this case 
it is a round hemisphere, and as we did before, it can be shown then that $A\left(t\right)\rightarrow 0$
as $t\rightarrow T$. If this blow up limit is non compact it must be the cigar.  In this case, it is not
difficult to prove that we can take as an origin for the blow up limit the
center of $D$. To study this case we define the following quantities which depend on $g\left(t\right)$:
\[
I\left(r,t\right)= \frac{L_g\left(\partial D_r\right)^2}{A_g\left(D_r\right)},
\quad \overline{I}\left(t\right)=\inf_{0\leq r\leq \mathcal{R}} I\left(r,t\right),
\]
where $D_r\subset D$ is the geodesic ball of radius $r$ centered at the center of $D$, $L_g\left(\partial D_r\right)$ is the length of $\partial D_r$ 
and $A_g\left(D_r\right)$ is the area of $D_r$ both with respect to the metric $g\left(t\right)$. 
There are two cases to be considered.
 The first case to consider is when there is a $\delta>0$ such that this infimum is attained at $r\left(t\right)\in \left[0,\mathcal{R}\right)$
on the time interval $\left(T-\delta, T\right)$.
If $r=0$, then $I=4\pi$;
if $r>0$, we have the following formula.
\begin{lemma}
\label{isoperimetricevolution}
If $\overline{I}$ at time $t$ is attained at $r\in\left(0,\mathcal{R}\right)$ then $I$ satisfies an evolution equation
\begin{equation}
\label{isoperimetric}
\frac{\partial}{\partial t}\log I\left(r,t\right) = \frac{\partial^2}{\partial r^2}\log I\left(r,t\right) +
\frac{1}{A_g\left(D_r\right)}\left(4\pi - I\left(r,t\right)\right).
\end{equation}
\end{lemma} 
\begin{proof}
We let $L=L_g\left(\partial D_r\right)$, $A=A_{g}\left(D_r\right)$, $k$ be the geodesic curvature of $\partial D_r$
and  $K$ be the Gaussian curvature of $D$.
We have the following set of formulas
\[
\frac{\partial L}{\partial r}=\int_{\partial D_r} k \,ds=kL,\quad \frac{\partial ^2 L}{\partial r^2}=-\int_{\partial D_r} K\, ds=\frac{\partial L}{\partial t},
\]
and
\[
\frac{\partial A}{\partial r}=L, \quad \frac{\partial^2 A}{\partial r^2}=\int_{\partial D_r} k\,ds, \quad
\frac{\partial A}{\partial t}=-4\pi+2\int_{\partial D_r} k\,ds.
\]
Notice that at the value of $r$ where the infimum is attained we have
\[
0=\frac{\partial}{\partial r}\log I=\frac{2}{L}\frac{\partial L}{\partial r}-\frac{1}{A}\frac{\partial A}{\partial r},
\quad \mbox{so we have,}\quad \frac{2}{L}\frac{\partial L}{\partial r}=\frac{1}{A}\frac{\partial A}{\partial r}.
\]
Formula (\ref{isoperimetric})
now follows from all these identities by a straightforward calculation.
\end{proof}
Clearly Lemma \ref{isoperimetricevolution} precludes the fact that $\overline{I}\rightarrow 0$, since it does imply that 
$\overline{I}$ increases if
$\overline{I}<4\pi$. Therefore, the blow up limit in this case cannot be the cigar, hence it is a round hemisphere, and we would be done.
We are left with one more possibility: the infimum is achieved  at $r=\mathcal{R}$ for a sequence of times
of times $t_k\rightarrow T$; if this is so, then since for the cigar $\overline{I} = 0$, if $A\left(t\right)\not\to 0$
as $t\rightarrow T$, we must have
$L\left(\partial D\right)\rightarrow 0$ as $t\rightarrow T$. Under the hypotheses of Theorem \ref{maintheorem2}, using the Maximum Principle, it is not difficult
to show that the scalar curvature remains uniformly bounded from below
on $\left(0,T\right)$, and therefore $\mathcal{R}$ is uniformly bounded  above 
in $0<t<T<\infty$. But then we have the following lemma (see \cite{Cortissoz1}).
\begin{lemma}
\label{lengtharea}
Let $g_k$ be a sequence of rotationally symmetric metrics on the two-ball $D$. Assume that there is a constant $\epsilon>0$ such that
$R_{g_k}\geq -\epsilon$ and
$k_{g_k}\geq -\epsilon$, and that the radius of $D$ with respect with this sequence of metrics is uniformly bounded from above
by $\rho>0$. 
Then if the length of the boudary of $D$, $L_{g_k}\left(\partial D\right)$, goes to $0$ as $k\rightarrow \infty$, then
$A_{g_k}\left(D\right)\rightarrow 0$.
\end{lemma}
\begin{proof}
By rescaling we may assume that $\epsilon=1$. Hence a comparison argument shows that,
\[
A_{g_k}\left(D\right)\leq 2\pi\int_0^{\rho}L_{g_k}\left(\partial D\right) e^{r}\,dr\leq 2\pi e^{\rho}\rho L_{g_k}\left(\partial D\right),
\]
and the conclusion of the lemma follows.
\end{proof}
The previous lemma shows that then we must have $A\left(t\right)\rightarrow 0$ as $t\rightarrow T$. This proves that we have, in any case,
$A\left(t\right)\rightarrow 0$ as $t\rightarrow T$ and the theorem follows. 

\appendix 
\section*{Appendix}
\appendix 

\section{Derivative estimates}
\label{curvaturebounds}
Let $\left(M,g\left(t\right)\right)$ be a solution to the Ricci flow (\ref{flow}).
In this appendix we will show how to produce bounds on the derivatives of the
curvature of $g\left(t\right)$ in terms of bounds on the curvature and on the geodesic curvature
of the boundary and its derivatives. The ideas we use are quite standard, as we produce certain quantities
involving derivatives of $R$ (clearly inspired
by the quantities used in the case of closed manifolds), 
and then we compute some differential inequalities; we will have to make computations
on the boundary of $M$ in order to apply the Maximum Principle to these 
differential inequalities, and even
though a bit tedious, these computations are certainly straightforward.

We now fix some notation.
Let $\rho\left(P,t\right)$ be the distance function to the boundary of $M$ with respect to the metric $g\left(t\right)$. We define the set
\[
M\left[0,\delta\right]=\rho^{-1}\left(\left[0,\delta\right],0\right),
\]
and we will refer to it as the collar of the boundary, or simply as the collar. We also will use
the notation,
\[
M_{c}=\rho^{-1}\left(\left\{c\right\},0\right)
\]
for the level curves of $\rho\left(\cdot,0\right)$.
All the quantities and symbols below (as for instance $\nabla$) depend on the time varying metric $g$, however
we will not use any subindex to indicate such dependence. The symbol
$\overline{\nabla}$ represents covariant differentiation in the boundary with respect to the induced metric.

Also, we must point out that
similar estimates can be obtained for solutions to the normalized Ricci flow (\ref{normalizedflow}),
either by working with it directly or using the normalization to pass back to the unnormalized flow,
but we will leave that to the reader.

Finally, recall that
 $\nu$ is the normal bundle of $\partial M$, and we denote by $\nu^{-}$ denote the bundle of inward pointing
normal vectors. We define the exponential map
\[
\exp: \nu^{-}\longrightarrow M,
\]
as
\[
\exp\left(p\right)=\gamma_p\left(1\right),
\]
where $\gamma_p$ is a geodesic starting at $p$ whose velocity vector is normal to $\partial M$ and
points inwards. For a compact surface, there is an $\delta>0$ for which this map is well defined when
restricted to
normal vectors to $\partial M$ of length at most $\delta$, so we will denote by $\nu^{-}\left(\delta\right)$
this subset of $\nu^-$.
Without much further ado, let us start with our estimates.
 
\subsection{First order derivative estimates}
Our purpose now is to show an estimate
on the first derivative of the curvature near the boundary
of $M$. We shall assume for simplicity that $k_g'=0$, but
it will be clear from the proof that if we have bounds
on $k_g'$ and its derivatives on any interval of time, with a judicious modification,
the following results are still valid. Notice the similarity with the analogous local interior estimate (see, for instance,
\cite[Theorem ~13.1]{Hamilton2} and \cite[Theorem ~14.14]{Chow2}).
\begin{theorem}
\label{firstderivative}
Let $\left(M,g\left(t\right)\right)$, $M$ compact, be a solution to the Ricci flow (\ref{flow})
on $\left[0,T^*\right]$. 
Let $\epsilon>0$ be such that at $t=0$, 
\[
\exp: \nu^{-}\left(\delta\right)\longrightarrow M\left[0,\delta\right]
\]
is a diffeomorphism. Let $K,\alpha>0$ be such that $\left|R\right|\leq K$ on $M\times\left[0,T^*\right]$ and
 $\left|k_g\right|,\left|\overline{\nabla}k_g\right|\leq \alpha$ on $\partial M\times\left[0,T^*\right]$. Then there is a 
$\tau:=\tau\left(K,\alpha\right)$ so that  
we can estimate
\[
\left|\nabla R\right|^2 \leq \frac{C\left(K,\alpha, \delta, T^*\right)}{t},
\quad
\mbox{on}\quad M\left[0,\delta\right]\times\left(0,\tau\right],
\]
where $C\left(K,\alpha, \delta, T^*\right)$ is a constant that depends on the given parameters.
\end{theorem}

\begin{proof}
Define on $M\left[0,\delta\right]$ for $t\in \left[0,T^*\right]$ the function
\[
F=te^{\alpha \rho}\left|\nabla^{\top}R\right|^2 + AR^2+B\rho K^2,
\]
where $\rho$ is the (time dependent) distance to the boundary function, $\nabla^{\top} R$ is the component of the gradient of $R$ 
with respect to $g\left(t\right)$ which is tangent 
to the
level surfaces of $\rho\left(\cdot,0\right)$ (at time $t$ with respect to $g\left(t\right)$), and $A$ and $B$ are positive constants.
 $F$ satisfies a differential inequality in $M$,
namely
\begin{eqnarray*}
\frac{\partial F}{\partial t}&\leq& \Delta_g F + 
\left(cte^{\alpha \rho}R+2\alpha^2te^{\alpha\rho} +e^{\alpha\rho}+t\left|\frac{\partial e^{\alpha \rho}}{\partial t}\right|
+t\left|\Delta_g e^{\alpha \rho}\right|-2A\right)\left|\nabla R\right|^2\\
&&-2\alpha\nabla \rho \nabla F +2\alpha\nabla\rho\left(2AR\nabla R + 2AK^2\nabla \rho\right)\\
&&+2AR^3
+\left|\frac{\partial \rho}{\partial t}\right| B K^2+\left|\Delta_g \rho\right| B K^2.
\end{eqnarray*}
where $c$ is a constant that can be computed, but whose actual value is
irrelevant. Now we need to control a few quantities from the previous expression.
The term $AR\nabla R$ can be dealt with via the inequality
\[
\left|AR\nabla R\right|\leq \frac{1}{2}\left(A^2R^2+\left|\nabla R\right|^2\right).
\] 
To control the term $\Delta_g \rho$ we use the identity 
\[
\Delta_g \rho\left(P\right)= -k_g\left(P\right),
\]
where $k_g\left(P\right)$ is the geodesic curvature of the level surface of $\rho$ that passes through
$P$. On the other hand the geodesic curvature of the level
curves of $\rho$ can be controled in terms of the curvature, a bound on the
geodesic curvature of $\partial M$, and $\epsilon>0$, since we have an equation
\[
\frac{\partial k_g}{\partial \rho}=\frac{R}{2}+k_g^2.
\]
Also, it is not difficult to obtain the estimate $\left|\frac{\partial \rho}{\partial t}\right|\leq K\rho$,
 and we have $\left|\nabla \rho\right|=1$. Now, by a convenient choice of $\tau\left(K,\alpha\right)>0$ and $A>0$, 
we obtain a differential inequality for $F$, valid on $0<t \leq \tau\left(\alpha,K\right)$,  namely
\[
\frac{\partial F}{\partial t}\leq \Delta_g F -2\alpha\nabla\rho\nabla F+C\left(A,B,\alpha\right)K^3.
\]
From now on, we will use the convention that $C$ is a constant (that may change from estimate to estimate)
that depends on, or any subset of, $A, B, \epsilon, K, \alpha, T^*$.

On the other hand, on $\partial M$, we have the identity (here $e_1$ is a unit
vector tangent to $\partial M$, and $e_2$ is the unit outward normal)
\[
\frac{\partial F}{\partial \eta_g}=-2t\alpha e^{\alpha\rho}\left|\nabla^{\top} R\right|^2+
2te^{\alpha\rho}k_g \left|\nabla^{\top}R\right|^2+
2te^{\alpha\rho}\left|\overline{\nabla}k_g\right|\left|\nabla^{\top} R\right|
+2Ak_gR^2-BK^2,
\]
where we have used the fact that 
\begin{eqnarray*}
\dfrac{\partial}{\partial \eta_g}\left|\nabla^{\top}R\right|^2&=& \nabla_{e_2}\left<\nabla_{e_1}R,\nabla_{e_1}R\right>\\
&=&2\left<\nabla_{e_2}\nabla_{e_1}R,\nabla_{e_1}R\right>\\
&=&2\left<\nabla_{e_1}\nabla_{e_2}R,\nabla_{e_1}R\right>\\
&=&2\left<\nabla_{e_1}\left(k_gR\right),\nabla_{e_1}R\right>\\
&=&2k_g\left|\nabla^{\top}R\right|^2+2\nabla_{e_1}k_g\cdot\nabla_{e_1}R,
\end{eqnarray*}
which gives an estimate
\[
\dfrac{\partial}{\partial \eta_g}\left|\nabla^{\top}R\right|^2\leq 
2k_g\left|\nabla^{\top}R\right|^2+2\left|\overline{\nabla}k_g\right|\cdot\left|\nabla^{\top}R\right|.
\]
So by taking $B\geq 2A\alpha+\dfrac{2\alpha T^*}{K}$ we have that
$
\displaystyle
\frac{\partial F}{\partial \eta_g}\leq 0.
$

In the part of the boundary of the collar that lies in the interior of the manifold, i.e. $M_{\epsilon}$, by Shi's interior estimates
(\cite[Theorem ~6.9]{Chow3}), we have,
\begin{eqnarray*}
F&\leq& C\left(A,B,\alpha\right)K^3 +C te^{C\left(K,\alpha,\epsilon\right)}K^2\left(\frac{1}{\delta^2}+\frac{1}{t}+K\right)\\
&\leq& C\left(A,B,\alpha\right)K^3+Ce^{C\left(K,\alpha,\delta\right)}K^2\left(\frac{T^*}{\delta^2}+1+KT^*\right).
\end{eqnarray*}
Applying the Maximum Principle yields
\[
F\leq \max_{t=0} F+tC\left(A,B,\alpha\right)K^3+
C\left(A,B,\alpha\right)K^3 + Ce^{C\left(K,\alpha,\delta\right)}K^2\left(\frac{T^*}{\delta^2}+1+KT^*\right),
\]
from which we obtain
\[
\left|\nabla^{\top}R\right|^2\leq 
\frac{C}{t}\left[K^3+K^2\left(\frac{T^*}{\delta^2}+1+K\right)\right].
\]

Now, let $\nabla^{\perp} R$ be the part of the gradient of $R$ perpendicular at time $t$ 
to the level curves of $\rho\left(\cdot,0\right)$. Define
\[
G=t\left|\nabla^{\perp}R\right|^2+AR^2.
\]
Again we have that $G$ satisfies a differential inequality
\[
\frac{\partial G}{\partial t}\leq \Delta_g G+2AR^3;
\]
on the other hand using the expresion for $\frac{\partial R}{\partial \hat{\eta}_g}$ on
$\partial M$ and Shi's interior derivative estimates on $M_{\delta}$,
in both components of the boundary of the collar we have an estimate
\[
G\leq T^*\alpha^2 K^2+AK^2 + CK^2\left(\frac{T^*}{\delta^2}+1+K\right)+2AK^3.
\]
This gives an estimate for $G$ similar to the estimate obtained for $F$. This proves the theorem.
\end{proof}

{\bf Remark.} The reader must notice the following. As the Ricci flow on surfaces can be written as
$e^ug_0$, if we assume bounds on the curvature, this implies bounds on the conformal factor. Therefore,
for the purposes of this paper, a bound on $\overline{\nabla}k_g$ can be safely changed
for a bound on the (tangential component of the)  gradient of $k_g$ with respect to 
the fixed metric $g_0$.

\subsection{Higher derivative estimates}

Here the previous procedure does not seem to work. The reason for this is that we do not know how to compute
higher order covariant derivatives of quantities on the boundary with just the knowledge of the
value of these quantities at the boundary.

To overcome this difficulty, we shall work using local coordinates. Let us write as usual $g=e^{u}g_{0}$ for the solution on the Ricci flow on the surface:
We shall show how to bound derivatives of $u$ when restricted to a Fermi chart. 
So, let $\left(U,\varphi\right)$ be a Fermi chart, with $U=\left(-\epsilon,\epsilon\right)\times\left[0,\delta\right)$
\[
\varphi:\,U\longrightarrow M,
\]
where if $p=\varphi\left(x,s\right)$, then $s$ is the distance from $p$ to the boundary.
In Fermi coordinates the metric is written as
\[
g=dx^2+f\left(x,s\right)^2\,ds^2.
\]
Since the Laplace operator is written as
\[
\Delta = e^{-u}\dfrac{1}{\sqrt{\left|g_0\right|}}\partial_i\left(g_0^{ij}\sqrt{\left|g_0\right|}\partial_j\right),
\]
in Fermi coordinates we have
\[
\Delta_{g_0} = \dfrac{\partial^2}{\partial s^2}-k_{g_0}\left(x,s\right)\frac{\partial}{\partial s}
+\dfrac{f_x}{f}\dfrac{\partial}{\partial x}+\dfrac{1}{f^2}\dfrac{\partial^2}{\partial x^2},
\]
where $k_{g_0}$ represents the geodesic curvature 
of the curve at distance $s$ from the boundary with respect to $\dfrac{\partial}{\partial s}$.

The main idea we shall employ to obtain estimates on $u$ is to use classical regularity results (Schauder estimates) for parabolic equations
with oblique boundary conditions (\cite[Chapter IV]{Lady}).
To be able to use these results and start a bootstrapping argument, we need to have uniform control in time
over the H{\"o}lder norm of the partial derivatives of $u$ in the Fermi chart.  Our main tool to obtain
this control will be Theorem 4.1 in \cite{Lieberman}, so we must verify the structure conditions
imposed as assumptions in this theorem
for the following elliptic operator
\[
\Delta_{g_0}u-R_{g_0}-Re^{u}, 
\]
with boundary condition
\[
-\dfrac{\partial u}{\partial s}-2k_{g}e^{\frac{u}{2}}-2k_{g_0}.
\]
In order to verify these structure conditions, we must have control over $R$, $R_{g_0}$, and 
their first derivative in the chart, and on $k_{g_0}$, 
$k_g$, and certain H{\"o}lder norms of their first derivatives in the chart. We
will assume bounds on the $C^{k}$ seminorms of $g_0$
in a Fermi chart $\left(U,\varphi\right)$ (this
takes care of the needed control on $R_{g_0}$
and $k_{g_0}$), and we also are assuming bounds on $k_g$ and its derivatives on $\partial M\times\left(0,T\right)$, and
on $R$ and $\nabla R$ (via Theorem \ref{firstderivative}) on $M\times\left(0,T\right)$. Notice that the bound given for $\nabla R$, once we have control over $u$
and $g_0$,
gives control over $\partial R$ in the chart. Indeed,
If in the chart we have that 
\[
\Lambda^{-2}\delta_{ij}\leq {g_0}_{ij}\leq \Lambda^2\delta_{ij}, 
\]
and $\left|u\right|<M$, then
\[
\left|\partial_i R \right|\leq \left|e^{u}\left(g_0\right)_{ij}\nabla^j R\right|
\leq \Lambda^2e^M\left|\nabla R\right|.
\]
These assumptions are enough to verify the structure conditions in Theorem 4.1 in \cite{Lieberman}. We leave the verification of these structure conditions 
to the interested reader.

As before, we shall use the notation
\[
\left\|g\right\|_{C^{k}\left(U,\varphi\right)}=\sup_{x\in U}\sum_{\left|\beta\right|\leq k}\left|\partial^{\beta} g\left(x\right)\right|
\]
to denote the $C^k$-norm of the metric $g$ written in local coordinates, and a few times for a function $u$
\[
\left\|u\right\|_{C^k_{\tau}\left(U,\varphi\right)}
\]
to denote the same $C^{k}$-norm with $t=\tau$ fixed (as $u$ depends also on $t$).
The H{\"o}lder spaces $C^{k,\alpha}\left(U,\varphi\right)$ are defined as usual.

We will also make use of the parabolic H{\"o}lder spaces $H^{\left(l\right)}=H^{l,\frac{l}{2}}$ as defined in pages 6-9 in \cite{Lady}:

\vspace{.1in}
Let $Q_T=\Omega\times\left(0,T\right)$, $\rho_0>0$ and $0<\alpha<1$ fixed. Define
\[
\left<u\right>^{\left(\alpha\right)}_{x,Q_T}=\sup_{
\tiny{\begin{array}{c}\left(x,t\right),\left(x',t\right)\in \overline{Q}_T,\\
\left|x-x'\right|<\rho_0\end{array}}}
\dfrac{\left|u\left(x,t\right)-u\left(x',t\right)\right|}{\left|x-x'\right|^{\alpha}},
\]
\[
\left<u\right>^{\left(\alpha\right)}_{t,Q_T}=\sup_{
\tiny{\begin{array}{c}\left(x,t\right),\left(x',t\right)\in \overline{Q}_T,\\
\left|x-x'\right|<\rho_0\end{array}}}
\dfrac{\left|u\left(x,t\right)-u\left(x,t'\right)\right|}{\left|t-t'\right|^{\alpha}}.
\]
For an integer $j$ define
\[
\left<u\right>_{Q_T}^{\left(j\right)}=\sum_{2r+\left|\beta\right|= j}\sup_{Q_T}\left|\partial_t^{r}\partial_x^{\beta}u\left(x,t\right)\right|.
\] 
Given $l+\alpha$, $l$ a positive integer as before and $0<\alpha<1$, define
\begin{eqnarray}
\label{holdernorm}
&\left\|u\right\|_{H^{\left(l+\alpha\right)}\left(\overline{Q}_T\right)}& \notag\\
&=&\\
&\sum_{2r+\left|\beta\right|=l}\left<\partial_t^r\partial_x^{\beta}u\right>^{\left(\alpha\right)}_{x,Q_T}+
\sum_{0<l+\alpha-2r-\left|\beta\right|<2}\left<\partial_t^r\partial_x^{\beta}u\right>^{\left(\frac{l+\alpha-2r-\left|\beta\right|}{2}\right)}_{t,Q_T}+
\sum_{j=0}^{l}\left<u\right>_{Q_T}^{\left(j\right)}.&\notag
\end{eqnarray}
The dependence of these norms
on $\rho_0>0$ is not important in our case, as different choices
produce equivalent norms.
We can then define the space $H^{\left(l+\alpha\right)}\left(\overline{Q}_T\right)$ as the space of all functions
with finite norm (\ref{holdernorm}). Finally the set $H^{\left(l+\alpha\right)}\left(Q_T\right)$ as the space of functions
belonging to $H^{\left(l+\alpha\right)}\left(\overline{Q'}\right)$ for any subdomain $Q'$ such that $\overline{Q'}\subset Q$.
In our notation (for these parabolic spaces) we will suppress the dependence on $\varphi$.

We are ready to state an proof the following:
\begin{theorem}
\label{derivativesconformalfactor}
Write $g=e^{u}g_0$ on $M\times\left[0,T\right)$. 
Let $\left(U,\varphi\right)$ be a Fermi chart with respect to $g_0$ of size $\left(\epsilon,\delta\right)$.
Assume that $u$ and $R$ are uniformly bounded on
$U\times\left(0,T\right)$, and that we have bounds for all integers $l>0$ and $0<\alpha<1$ on 
$\left\|k_g\right\|_{H^{l+\alpha}\left(\left(\left(-\epsilon,\epsilon\right)\times\left\{0\right\}\right)\times\left(0,T\right)\right)}$ 
(recall $k_g=\psi$ in (\ref{flow})). 
Let $\tau$ be such that $0<\tau<T$, and $U'$ be an open subset of $U$ such that $\overline{U'}\subset U$.
Then
for each integer $k>0$ there is a $C_k$ which depends only on the bound on $R$, on the bounds on $k_g$ and its
derivatives,
on bounds on $\left\|g_0\right\|_{C^{l}\left(U,\varphi\right)}$, on $U'$ and on $\tau$ and $T$, such that
\[
\left|\partial^k u\left(x,t\right)\right|\leq C_k \quad \mbox{on}\quad U'\times\left[\tau,T\right),
\]
where $U'=\left(-\frac{\epsilon}{2},\frac{\epsilon}{2}\right)\times \left[0,\frac{\delta}{2}\right)$ and $\partial$ 
represents partial differentiation in the chart.
\end{theorem}

\begin{proof}

Using the chart, we shall work now in the set $U\subset \mathbb{R}_+^2$. Define the following sets
\[
U^{k}=\left(-\epsilon\left(\dfrac{1}{2}+\dfrac{1}{2^{k+1}}\right),\epsilon\left(\dfrac{1}{2}+\dfrac{1}{2^{k+1}}\right)\right)\times 
\left[0,\delta\left(\dfrac{1}{2}+\dfrac{1}{2^{k+1}}\right)\right),
\]
\[
\partial'U^{k}=
\left(-\epsilon\left(\dfrac{1}{2}+\dfrac{1}{2^{k+1}}\right),\epsilon\left(\dfrac{1}{2}+\dfrac{1}{2^{k+1}}\right)\right)
\times \left\{0\right\},
\]
and 
\[
W^k=U^k\times\left(\tau_k,T\right),
\quad \tau_k=\tau\left(\dfrac{1}{2}-\dfrac{1}{2^{k+1}}\right).
\]
Fix $\tau_1>0$. Then for any $t\geq \tau_1$,
having control over $R_{g_0}$ and $\partial R_{g_0}$ (as we are assuming this control over the metric $g_0$),
and over $R$ and $\nabla R$ (and this control is uniform in $t$ for $0<\tau_1\leq t<T$ once we have fixed $\tau_1$)
from the elliptic equation
\[
\Delta_{g_0}u-R_{g_0}=Re^u
\]
with oblique boundary condition
\[
\dfrac{\partial u}{\partial \eta_{g_0}}=2k_{g}e^{\frac{u}{2}}-2k_{g_0},
\]
we can obtain an estimate on the $C^{1,\alpha}$ norm of $u$ at any time $t>0$ 
(see Theorem 4.1 in \cite{Lieberman} and the paragraph right after the statement of
the theorem). 
To be more precise, we obtain, for any $t\geq\tau_1 >0$, a bound $C_{1,\alpha}$ (uniform in $t$)
\begin{eqnarray*}
&\left\|\partial_k u\right\|_{C_t^{\alpha}\left(U^{1},\varphi\right)}&\\
&\leq&\\
& C_{1,\alpha}\left(U,U^{1},\mu_0,\mu_1,
\left\|k_g\right\|_{H^{\left(1+\alpha\right)}\left(\left(\left(-\epsilon,\epsilon\right)\times\left\{0\right\}\right)\times\left(0,T\right)\right)}
,\left\|g_0\right\|_{C^{2,\alpha}\left(U,\varphi\right)}
, \tau_1,T\right),&
\end{eqnarray*}
were $\mu_0$ is a bound on $R$ and $\mu_1$ is a bound on $\partial R$ on $U\times\left(\tau_1,T\right)$,
From Theorem \ref{firstderivative}, we can even suppress the dependence on $\mu_1$ as we can give it in terms of $\mu_0$, and 
on $\tau_1$ and $T$.
Also, from the bounds on $R$, and equation
\[
\dfrac{\partial u}{\partial t}=-R,
\]
we find a bound on $\dfrac{\partial u}{\partial t}$. This shows that
$u\in H^{\left(1+\alpha\right)}\left(W^1\right)$, and that bounds in this parabolic H{\"o}lder space are controlled 
by bounds on $R$ and on $g_0$ and its derivatives. 

Now we use standard parabolic estimates for the problem (in particular Theorem 10.1 in \cite[Chapter IV]{Lady}: 
For the relevant definitions please consult also \cite[Chapter I-\S 1]{Lady})
\[
\left\{
\begin{array}{l}
\dfrac{\partial u}{\partial t}=e^{-u}\left(\Delta_{g_0} u-R_{g_0}\right) \quad \mbox{in}\quad U\times\left(0,T\right)\\
\dfrac{\partial u}{\partial \eta}=2k_{g}e^{\frac{u}{2}}-2k_{g_0}\quad \mbox{on}\quad \left(\left(-\epsilon,\epsilon\right)\times\left\{0\right\} 
\right)\times\left(0,T\right).
\end{array}
\right.
\]
As we have $H^{\left(1+\alpha\right)}$ bounds on the boundary terms (here we use again our assumptions on $k_g$),
this gives us a bound on $\left\|u\right\|_{H^{\left(2+\alpha\right)}\left(W^2\right)}$ and of course this bound depends on all
the previous bounds, namely 
\begin{eqnarray*}
&\left\|u\right\|_{H^{\left(2+\alpha\right)}\left(W^2\right)}&\\
&\leq&\\
& C\left(W^1,W^2,\left\|u\right\|_{H^{\left(1+\alpha\right)}\left(W^1\right)},\left\|k_{g}\right\|_{H^{\left(1+\alpha\right)}\left(\left(\partial'U^1\right)\times\left(\tau_1,T\right)\right)},
\left\|g_0\right\|_{C^{2,\alpha}\left(U,\varphi\right)}\right)& 
\end{eqnarray*}
From this bound on $\left\|u\right\|_{H^{\left(2+\alpha\right)}\left(W^2\right)}$ we can obtain 
 bounds on $\left\|u\right\|_{H^{3+\alpha}\left(W^3\right)}$, and this bounds also depend only on $R$. 
 We can continue this process (bootstrapping), by repeated applications of Theorem 10.1 in in \cite[Chapter IV]{Lady}, to obtain bounds
 on higher H{\"o}lder norms of $u$. 
 Indeed, given a bound on $\left\|u\right\|_{H^{\left(n+\alpha\right)}\left(W^n\right)}$, \cite[Theorem 10.1, Chapter IV]{Lady} gives us a bound
 \begin{eqnarray*}
&\left\|u\right\|_{H^{\left(n+2+\alpha\right)}\left(W^{n+1}\right)}&\\
&\leq&\\ 
&C\left(W^n,W^{n+1},\left\|u\right\|_{H^{\left(n+\alpha\right)}\left(W^n\right)},
\left\|k_{g}\right\|_{H^{\left(n+1+\alpha\right)}\left(\left(\partial' U^n\right)\times\left(\tau_n,T\right)\right)},
\left\|g_0\right\|_{C^{n+2,\alpha}\left(U,\varphi\right)}\right).&
\end{eqnarray*}
So we finally can bound any number 
of derivatives of $u$ in terms of bounds on $R$, on $k_g$ and its derivatives, and on $g_0$ and its derivatives in the chart $\left(U,\varphi\right)$
(in ever smaller domains, but this is not a problem, as all this domains
contain $U'\times\left[\tau,T\right)$). This proves the Theorem.
\end{proof}

\section{On the double of a surface}

Here we discuss the regularity of the metric of the double of a manifold.
Given $M$ a manifold with boundary, we define its double as follows:

We let $M_j=M\times\left\{j\right\}$, $j=0,1$, and define an equivalence relation
\[
\left(m,i\right)\sim \left(m,j\right)
\]
if $i=j$, or if $i\neq j$ and $m\in \partial M$.
Then the double is the set
\[
\tilde{M}= \left(M_0\cup M_1\right)/\sim
\]
endowed with the quotient topology. To give a smooth structure to $\tilde{M}$, given a chart
\[
\psi: U\subset \mathbb{R}_{+}^n\longrightarrow M
\]
we define a chart on $\tilde{M}$ as follows.
If $U\cap \left\{x^n=0\right\}=\emptyset$, we obtain two charts
by defining
\[
\tilde{\psi}_j=\left(\psi\left(x\right),j\right).
\]
If $U\cap \left\{x^n=0\right\}\neq\emptyset$, let $U^*\subset \mathbb{R}_-^n$ be defined as
\[
U^*=\left\{x:\, \left(x^1,\dots,x^{n-1},-x^n\right)\in U\right\},
\]
and define
\[
\tilde{\psi}:\,U\cup U^*\longrightarrow \tilde{M}
\]
as 

\[
\tilde{\psi}\left(x^1,\dots,x^n\right)=
\left\{
\begin{array}{l}
\left(\psi\left(x\right),0\right) \quad \mbox{if}\quad x^n\geq 0\\
\left(\psi\left(x\right),1\right) \quad \mbox{if}\quad x^n\leq 0
\end{array}
\right.
\]
This gives a smooth structure to $\tilde{M}$. We call $\tilde{M}$ the double of $M$

However, when $M$ is also endowed with a metric, 
the regularity of the double metric is a different matter, and in general it might not be smooth. 
Let us first review how to double the metric. We shall explain the
construction near the boundary, being relatively obvious the construction away from the boundary.

Pick a Fermi chart on $M$,
\[
\varphi:\left(-\epsilon,\epsilon\right)\times \left[0,\delta\right)\longrightarrow M,
\]
and as explained above, construct the double chart. In this new chart the metric 
is written as
\[
g=ds^2+f\left(x,s\right)^2 dx^2, \quad \mbox{if}\quad s\geq 0,
\]
and
\[
g=ds^2+f\left(x,-s\right)^2 dx^2, \quad \mbox{if}\quad s\leq 0.
\]
We shall write
\[
g_{11}=1,\quad g_{22}= f^2, \quad g_{12}=0=g_{21}.
\]
Notice that at $\partial M$, $g$ is the euclidean metric (i.e. $f\left(x,0\right)=1$). So we have obtained at least 
a continuous metric on the double. However, if we assume that the original metric is smooth
we can show the following:

\begin{proposition}
\label{metricdouble}
Assume that $k_g=0$, then the metric on the double is at least $C^{2,1}$. Here
$C^{2,1}$ is the space of functions twice differentiable and whose second
derivatives
are locally (i.e., in the chart) Lipschitz continuous.
\end{proposition}

\begin{proof}
Let $\left(U,\varphi\right)$ be a Fermi chart
around a point $p\in \partial M$. Then the metric is written as a $2\times 2$ 
positive definite matrix $g_{ij}$, as shown above. All we must prove is that when doubled, each of these
functions is $C^{2,1}$. 
The fact that $k_g=0$ implies that the metric is $C^2$
(see \cite{Escobar}, the argument in pp 43-44), and as we already know that the $g_{ij}$'s are smooth away 
from $\partial M$ (being
the original metric smooth),
the proposition will follow from Lemma \ref{Lipschitz}.
\end{proof}

Before stating and proving Lemma \ref{Lipschitz}, let us introduce some notation. Given a point $x=\left(x^1,x^2\right)$,
define
\[
\left\|x\right\|=\sqrt{\left(x^1\right)^2+\left(x^2\right)^2},
\]
and for a fixed $r>0$ we let
\[
B_{+}=\left\{x\,| \left\|x\right\|<r, \quad x^2\geq 0\right\},
\] 
and in the same way define $B_-$. Given $x=\left(x^1,x^2\right)$ define $x^{*}=\left(x^1,-x^2\right)$.
For a pair of functions
$f_+:B_{+}\longrightarrow \mathbb{R}$ and $f_-:B_{-}\longrightarrow \mathbb{R}$ which coincide
over $\left\{x^2=0\right\}$ define
\[
f_+\cup f_-:B_{+}\cup B_{-}\longrightarrow \mathbb{R}
\]
as 
\[
f_+\cup f_-\left(x\right)= f_{\pm}\left(x\right)\quad
\mbox{if}\quad x\in B_{\pm}.
\]
Then we have the following elementary lemma.
\begin{lemma}
\label{Lipschitz}
Let $f_+:B_+\longrightarrow \mathbb{R}$ and $f_- :B_-\longrightarrow \mathbb{R}$ be Lipschitz .
Asume that $f_{+}=f_{-}$ on $B_{+}\cap\left\{x^2=0\right\}$. Then $f_+\cup f_{-}$ is also Lipschitz.
\end{lemma}
\begin{proof}
Let $x=\left(x^1,x^2\right)\in B_-$ and $y=\left(y^1,y^2\right)\in B_+$. Then we have the following inequalities
\begin{eqnarray*}
\left|\tilde{f}\left(x\right)-\tilde{f}\left(y\right)\right|&=&\left|\tilde{f}\left(x^1,x^2\right)-\tilde{f}\left(y^1,y^2\right)\right|\\
&=& \left|f\left(x^1,0\right)-f\left(x^1,x^2\right)\right|+\left|f\left(x^1,0\right)-f\left(y^1,0\right)\right|\\
&&
+\left|f\left(y^1,0\right)-f\left(y^1,y^2\right)\right|\\
&\leq& C\left|x^2\right|+C\left|y^2\right|+C\left|y^1-x^1\right|\\
&=& C\left(-x^2+y^2\right)+C\left|x^1-y^1\right|\\
&=& C\left|y^2-x^2\right|+C\left|x^1-y^1\right|\\
&\leq& 2C\left\|y-x\right\|.
\end{eqnarray*} 
\end{proof}

\subsection{An application: doubling and regularity for the Ricci flow} 
\label{regularityricciflow}
Let $g$ be a solution to the Ricci flow on $M\times\left(A,B\right)$ with $k_g\equiv 0$. Let 
$\tilde{M}$ be the double of $M$, then from the solution $g$ we obtain in an obvious way
a solution $\tilde{g}$ of the Ricci flow on $\tilde{M}\times\left(A,B\right)$. Let us show that 
$\tilde{g}$ is smooth. Pick $t_0\in\left(A,B\right)$; by Proposition \ref{metricdouble} we know 
that $\tilde{g}\left(t_0\right)$ is $C^{2,1}$, and smooth away from $\partial M$ (in this case we shall
refer as $\partial M$ to
the subset of $\tilde{M}$ corresponding to the equivalence class of $\partial M\times\left\{0\right\}$).

Consider $p\in \partial M \subset \tilde{M}$. Let $R_{\tilde{g}\left(t_0\right)}$ be the curvature of the double
at time $t=t_0$. By
our considerations, $R_{\tilde{g}\left(t_0\right)}$ is Lipschitz. Then the equation
\[
\Delta_{\tilde{g}\left(t_0\right)}u_0=R_{\tilde{g}\left(t_0\right)},
\]
has a $C^{2,\alpha}$ solution on a perhaps even smaller neighbourhood of $p$
(see \cite[Theorem 2.3]{DeTurck}). 
Notice that $e^{u_0}\tilde{g}\left(0\right)$ is flat, and hence in a coordinate system around $p$ 
we have that $e^{u_0}\tilde{g}\left(0\right)$ can be written as the euclidean metric, that we will denote again 
by $g_E$
 (another way of proving this is by using the existence of
isothermal coordinates see \cite{Chern}).
Hence, we have a solution to the Ricci flow $\tilde{g}$ with initial condition $\tilde{g}\left(t_0\right)$, 
so we have a solution, in a small neighbourhood of $p$,
to (writing $\tilde{g}=e^ug_E$)
\[
\dfrac{\partial u}{\partial t}=e^{u}\Delta_{g_E}u, \quad u\left(\cdot, t_0\right)=u_0,
\]
and $u_0$ is at least $C^{2,\alpha}$, for any $0<\alpha<1$. 
Therefore by parabolic regularity (\cite[Theorem 10.1, Chapter IV]{Lady}), $u$ is smooth, and so is the metric
$
\tilde{g}= e^{u}g_{E},
$
since $g_E$ is smooth. This means that the solution to the Ricci flow becomes smooth for $t>t_0$. 
This justifies, at least in the case of surfaces the doubling procedure: the solution
to the Ricci flow becomes smooth instantaneously, so all the results on 
long time behaviour proved for smooth solutions apply. In particular, if 
we double
an ancient solution along its totally geodesic boundary, we obtain a smooth ancient solution to the Ricci flow.

\section{A compactness result}

The results in this appendix can be found in \cite{Anderson}, Section 3.1 (see the references therein). We rewrite
them here for the convenience of the reader.

Given $s=l+\sigma$ ($l\in \mathbb{Z}^+$ and $0<\sigma<1$), $\rho>0$ and $Q\in\left(1,2\right)$, let $\mathcal{N}$
denote the class of connected Riemannian manifolds $\left(M,g\right)$ with boundary with the following properties.

(A) If $dist\left(p,\partial M\right)>\rho$, there is a neighbourhood $V$ of $p$ contained in the interior of $M$
and a coordinate chart 
\[
\varphi:B_{\frac{\rho}{2}}\left(0\right)\longrightarrow U, \quad \varphi\left(0\right)=p,
\]
such that in these coordinates
\begin{equation}
\label{condition1}
Q^{-2}\delta_{ij}\leq g_{ij}\leq Q^{2}\delta_{ij},
\end{equation}
and
\begin{equation}
\label{condition2}
\rho^s\sum_{\left|\beta\right|=l}\sup\left|x-y\right|^{-\sigma}\left|g_{ij}\left(x\right)-g_{ij}\left(y\right)\right|\leq Q-1
\end{equation}

(B) If $dist\left(p,\partial M\right)\leq \rho$ there is a neighbourhood $U$ of $p$ and a coordinate chart
\[
\varphi:B^+_{4\rho}\left(0\right)\longrightarrow U, \quad \varphi\left(0\right)=p,
\]
such that $\left\{x_2=0\right\}$ maps to $\partial M$ and (\ref{condition1})-(\ref{condition2}) hold in these coordinates.

Let $\mathcal{N}_{*}\left(s,\rho,Q\right)$ denote the class of pointed manifolds $\left(M,g,p\right)$ with $p\in M$,
satisfying these properties. Then we have the following compactness result.
\begin{theorem}
\label{fundamentalcompactness}
Given $s,\rho\in\left(0,\infty\right)$, the class $\mathcal{N}_{*}\left(s,\rho,Q\right)$ is compact
in $\mathcal{N}_{*}\left(s',\rho,Q\right)$ in the pointed $C^{s'}$-topology for all $s'<s$.
\end{theorem}
 
\end{document}